\documentclass[11pt]{article}

\usepackage{amsmath,verbatim}
\usepackage{amsfonts}
\usepackage{amssymb}
\usepackage{amsthm,epic}
\usepackage{mathrsfs}
\usepackage{xcolor}
\usepackage{hyperref}

\usepackage{tikz} 
\usepackage[T1]{fontenc}
\usepackage{enumitem}
\usepackage{pgf, tikz, pstricks, pict2e}
\usepackage{graphicx, xcolor}
\usepackage{subcaption}
\usepackage{footnote,pifont,dsfont,mathrsfs}
\usetikzlibrary{patterns}

\renewcommand{\geq}{\geqslant}
\renewcommand{\leq}{\leqslant}
\renewcommand{\epsilon}{\varepsilon}

\newtheorem{thm}{Theorem}

\newtheorem{lem}[thm]{Lemma}

\theoremstyle{definition}
\newtheorem{defn}[thm]{Definition}

\theoremstyle{remark}

\newtheorem{rem}[thm]{Remark}
\usepackage{array}
\usepackage[font={small}]{caption}

\theoremstyle{remark}

\usepackage[font={small,it}]{caption}

\colorlet{darkgreen}{green!50!black}
\definecolor{darkseagreen}{rgb}{0.56, 0.74, 0.56}
\definecolor{lightcyan}{rgb}{0.88, 1.0, 1.0}
\definecolor{lightblue}{rgb}{0.68, 0.85, 0.9}
\definecolor{palecerulean}{rgb}{0.61, 0.77, 0.89}
\definecolor{lgreen} {RGB}{180,210,100}
\definecolor{dblue}  {RGB}{20,66,129}
\definecolor{ddblue} {RGB}{11,36,69}
\definecolor{lred}   {RGB}{220,0,0}
\definecolor{nred}   {RGB}{224,0,0}
\definecolor{norange}{RGB}{230,120,20}
\definecolor{nyellow}{RGB}{255,221,0}
\definecolor{ngreen} {RGB}{98,158,31}
\definecolor{dgreen} {RGB}{78,138,21}
\definecolor{nblue}  {RGB}{28,130,185}
\definecolor{jblue}  {RGB}{20,50,100}
\definecolor{Apricot} {RGB}{255, 170, 123} 
\definecolor{dpurple}  {RGB}{53,21,93}

\usepackage{multicol}

\makeatletter
\def\testb#1{\testb@i#1,,\@nil}%
\def\testb@i#1,#2,#3\@nil{%
  \draw[->, thick] (O) --++(#1);
  \ifx\relax#2\relax\else\testb@i#2,#3\@nil\fi}
\makeatother

\usepackage[T1]{fontenc}
\usepackage[mathscr]{eucal}
\usepackage{fullpage}

\makeatletter
\def\testbb#1{\testbb@i#1,,\@nil}%
\def\testbb@i#1,#2,#3\@nil{%
  \draw (O) --++(#1);
  \ifx\relax#2\relax\else\testbb@i#2,#3\@nil\fi}
\makeatother

\usepackage{manyfoot}
\DeclareNewFootnote{N}[arabic]

\date{\today}
\author{Am\'elie Trotignon\thanks{Institute for Algebra, Johannes Kepler University, Linz, Austria; \texttt{amelie.trotignon@jku.at}.  A. Trotignon was supported by the Austrian Science Fund (FWF) grant FWF05004.}}

\title{Discrete harmonic functions in the three-quarter plane}

\begin{document}
\maketitle

\vspace{-4mm}

\begin{abstract}
In this article we are interested in finding positive discrete harmonic functions with Dirichlet conditions in three quadrants. Whereas planar lattice (random) walks in the quadrant have been well studied, the case of walks avoiding a quadrant has been developed lately. We extend the method in the quarter plane -- resolution of a functional equation via boundary value problem using a conformal mapping -- to the three-quarter plane applying the strategy of splitting the domain into two symmetric convex cones. We obtain a simple explicit expression for the algebraic generating function of harmonic functions associated to random walks avoiding a quadrant.
\end{abstract}

\section{Introduction}
\label{sub:intro}

In the two-dimensional continuous case, a continuous harmonic function is a function for which the standard Laplacian $\Delta=\frac{\partial^{2}}{\partial x^{2}}+\frac{\partial^{2}}{\partial y^{2}}$ is zero. For example, in the cone $\{ r\exp(it): 0\leq r < \infty, 0\leq t\leq \eta\}$ of opening angle $\eta$,  then ${u(r,t)=r^{\pi/\eta}\sin \left( \frac{\pi}{\eta}t \right)}$ is the unique (up to multiplicative constants) positive harmonic function equal to zero on the boundary. 

In the two-dimensional discrete case, consider the simplest Laplacian operator 
\[
\frac{1}{4}[f(i+1,j)+f(i,j+1)+{f(i-1,j)}+f(i,j-1)]-f(i,j).
\] 
The function $f(i,j)=i$ is positive harmonic in the right half-plane and equal to zero on $y$-axis. In the positive quadrant, the function $f(i,j)=ij$ is positive harmonic with Dirichlet boundary conditions.  What about the three-quarter plane? Surprisingly, harmonic functions in the three-quarter plan have more complex expressions, see equations \eqref{thm:expression-D-3/4} and \eqref{thm:expression-H}. Unlike the two last examples, the three quadrants is not a convex cone and finding positive harmonic functions is more complicated. Although  the difference between the intersection of two half-planes (the quarter plane) and union of two half-planes (the three-quarter plane) seems geometrically insignificant at first sight, Bousquet-M\'elou has already noticed this complexification with a paths enumeration point of view: in \cite{BM-16} Bousquet-M\'elou notes that studying
the simple walks (with cardinal points jumps) is as difficult as studying the Gessel walks in the quadrant (a famous model with jumps $\{{\sf E},{\sf NE},{\sf W},{\sf SW}\}$, solved after ten years of research). In this article, we develop a systematic approach to find positive discrete harmonic functions in the three-quarter plane with Dirichlet conditions. 

\paragraph{Context.}
In the discrete case, planar lattice random walks in cones occupy a central position in probability and combinatorics. Harmonic functions play an important role in probability theory. Doob $h$-transform is a way to build conditioned random processes in cones from a random process and a positive harmonic function vanishing on the boundary of the cone. Finding positive harmonic functions for random processes is therefore a natural objective in the study of confined random walks. There are very few ways to compute discrete harmonic functions (see \cite{Ra-14} and references therein). 

Most of walks study have been done on the quadrant, or more generally in convex cones.  A natural generalization is to consider other domains of restriction and determine how the framework of (random) walks is different from the quarter plane to this region. Recently, non-convex cones, in particular the three-quarter plane
\begin{equation}
\label{eq:3quarters}
     \mathcal C=\{(i,j)\in\mathbb Z^2: i\geq0 \text{ or } j\geq0\},
\end{equation}
have been examined. Unlike the quarter plane, where generating functions involve only positive powers of the variables, in the three-quarter plane we face both positive and negative powers of the variables, making convergence issues arise. A natural strategy consists of cutting the three quadrants in some convex cones in which the generating functions are convergent. In \cite{BM-16}, Bousquet-M\'elou sees the three-quarter plane as the union of three quadrants and obtains some results for the simple and diagonals walks avoiding a quadrant. Integral expressions for the generating function of walks avoiding a quadrant  with symmetric step sets for walks are derived in \cite{RaTr-18}, where the three-quarter plane is seen as the union of two symmetric convex cones of opening angle $3\pi/4$.  Asymptotics of the number of excursions of walks with small steps in the three-quadrant is computed in \cite{Mu-19} by Mustapha.  In this article, following \cite{DeWa-15, BoRaSa-14}, Mustapha expresses the critical exponent of harmonic functions in three quadrants as a function of the critical exponent of harmonic functions in a quadrant.  When this exponent is not rational, then the generating function of walks is not D-finite and \cite[Thm~1.3]{Mu-19} proves that the generating function of the walks of the 51 non-singular step sets with infinite group are not D-finite in the three-quarter plane (recall that neither are they in the quadrant).

In this article, we  find an explicit expression for generating functions of discrete harmonic functions associated to random walks avoiding a quadrant with a mixed approach of \cite{Ra-14} and \cite{RaTr-18}. We focus on the analytic approach developed in \cite{Ra-14}, which consists in writing a functional equation for the generating function for a fixed harmonic function, transforming this functional equation into a boundary value problem and finally solving this problem, which results in an explicit expression for the generating function. We begin by making some assumptions on both the random walks and discrete harmonic functions for these random walks. 

\begin{enumerate}[label=(H\arabic*)]
\item \label{H1}
The walk is homogeneous inside the cone with transition probabilities $\{p_{i,j}\}_{-1\leq i,j\leq 1}$ to the nearest neighbors; 
\item \label{H2}
The transition probabilities are symmetric ($p_{i,j}=p_{j,i}$) and $p_{0,0}=p_{-1,1}=p_{1,-1}=0$;
\item \label{H3}
In the list $p_{1,1}$, $p_{1,0}$, $p_{0,-1}$, $p_{-1,-1}$, $p_{-1,0}$, $p_{0,1}$, there are no two consecutive zeros;
\item \label{H4}
The drifts are zero: $\sum_{-1\leq i,j \leq 1}ip_{i,j}=0$ and $\sum_{-1\leq i,j \leq 1}jp_{i,j}=0$.
\end{enumerate}

We first suppose with \ref{H1} the random walks to be homogeneous with small steps, which therefore provide us techniques and tools developed in \cite{FaIaMa-17}. Moreover, with assumption \ref{H2}, we suppose the walks to be symmetric with no anti-diagonal jumps. If we do not exclude non-zero probabilities $p_{-1,1}$ and $p_{1,-1}$, as noted in \cite[Sec.\ 2]{RaTr-18}, after the change of variable $\varphi$ we would end up with random walks with non-zero probabilities $p_{-2,1}$ and $p_{2,1}$ respectively for big jumps of vector $(-2,-1)$ and $(2,1)$. Note that the analytic theory for walks with big steps is still incomplete \cite{CoBo-83,Co-92}. A related consequence would be the presence of more unknowns in the functional equation \eqref{eq:functional_equation_sym} (namely $\sum_{i\geq 1}f(i-2,i)x^{i}y^{i}$ and $\sum_{i\geq 1}f(i,i-2)x^{i}y^{i}$), making the problem not solvable with our techniques. The third hypothesis \ref{H3} is not essential in the study, but automatically excludes degenerate cases which could have been studied with easier methods.  Finally, we assume the walks to have zero drifts with \ref{H4}. Note that the zero and non-zero drifts are two very different frameworks \cite{Du-14}, and most results are given in the non-zero drift case. Furthermore, \ref{H4} makes the random walks hit the negative axes almost surely. Observe that the hypotheses \ref{H2} and \ref{H4} are non-redundant: a model can satisfy \ref{H2} and not \ref{H4}\footnoteN{For example the random walk with transition probabilities $p_{-1,0}=p_{0,-1}=1/4$ and $p_{1,1}=1/2$.} and vice versa\footnoteN{For instance the random walk with transition probabilities $p_{1,0}=p_{1,1}=p_{0,-1}=1/5$ and $p_{-1,0}=2/5$.}. Figure~\ref{fig:Walks-H1H2H3} illustrates a possible transition probabilities set which satisfies \ref{H1}, \ref{H2} and \ref{H3}. We will use the combinatorial step set terminology: the simple model is when $p_{1,0}=p_{1,0}=p_{-1,0}=p_{0,-1}=1/4$, Gouyou-Beauchamps model has $p_{1,0}=p_{-1,1}=p_{-1,0}=p_{1,-1}=1/4$ and Gessel model is $p_{1,0}=p_{1,1}=p_{-1,0}=p_{-1,-1}=1/4$. In addition, we ask the associated discrete harmonic functions $f=\left(f(i,j)\right)_{(i,j)\in\mathcal{C}}$ to satisfy four properties:

\begin{enumerate}[label=(P\arabic*)]
\item \label{P1} 
For all $i\geq 1$ or $j \geq 1$, $f(i,j)=\sum_{-1\leq i_{0},j_{0} \leq 1}p_{i_{0},j_{0}}f(i+i_{0},j+j_{0})$;
\item \label{P2} 
If $i\leq 0$, $f(i,0)=0$ and if $j\leq 0$, then $f(0,j)=0$;
\item \label{P3} 
If $i>0$ or $j>0$ then $f(i,j)>0$. 
\end{enumerate} 
We also make the hypothesis \ref{P4} which, as we shall see (Remark~\ref{rem:symm-cdt}), will be automatically satisfied for positive harmonic functions associated to symmetric step sets.
\begin{enumerate}[label=(P4)]
\item \label{P4}
For all $(i,j)\in\mathcal{C}$, $f(i,j)=f(j,i)$.
\end{enumerate}

The first property \ref{P1} is the harmonicity condition, the second one \ref{P2} is the zero condition on the boundary, the third one \ref{P3} is the positivity condition within the cone and the last one \ref{P4} is a symmetry condition (coming then from \ref{H2}). In other words, we are interested in symmetric positive discrete harmonic functions for symmetric random walks with small steps constrained to the boundary of $\mathcal{C}$. Thereafter, for such a fixed harmonic function $f$, we consider its generating function 
\begin{equation}
\label{eq:generating_function_3/4}
	H(x,y)=\sum_{(i, j)\in \mathcal C} f(i, j)x^{i-1}y^{j-1}.
\end{equation}
We observe here one of the main difficulties of the three quadrants: the series $H(x,y)$ is not convergent. Thereafter, we may see this object as a formal series. The goal is to find a formal explicit expression of $H(x,y)$, and from this expression, it is not difficult to get the $f(i,j)$ for all $(i,j)\in\mathcal{C}$, with a formal power series expansion. 

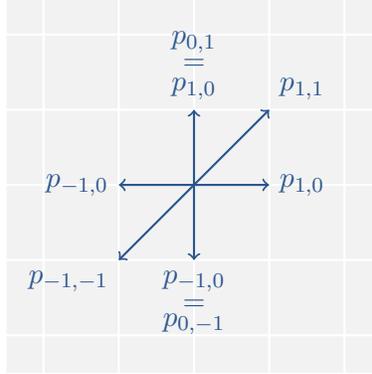
\begin{figure}[t]
\centering
\begin{tikzpicture}
\begin{scope}[scale=1, xshift=0 cm, yshift=0 cm]
\draw[white, fill=gray!10] (-2.5,-2.5) -- (-2.5,2.5) -- (2.5,2.5) -- (2.5,-2.5);
\draw[white, thick] (-2.5,-2.5) grid (2.5,2.5);
\draw[dblue!90, thick, ->] (0,0)--(1,0) node[ right]{$p_{1,0}$};
\draw[dblue!90, thick, ->] (0,0)--(-1,0) node[ left]{$p_{-1,0}$};
\draw[dblue!90, thick, ->] (0,0)--(0,-1) node[below ]{$p_{-1,0}$};
\draw[dblue!90, thick] (0,-1.6) node[]{$=$};
\draw[dblue!90, thick] (0,-1.85) node[]{$p_{0,-1}$};
\draw[dblue!90, thick, ->] (0,0)--(0,1) node[above ]{$p_{1,0}$};
\draw[dblue!90, thick] (0,1.6) node[ ]{$=$};
\draw[dblue!90, thick] (0,1.9) node[ ]{$p_{0,1}$};
\draw[dblue!90, thick, ->] (0,0)--(1,1) node[above right]{$p_{1,1}$};
\draw[dblue!90, thick, ->] (0,0)--(-1,-1) node[below left]{$p_{-1,-1}$};
\end{scope}
\end{tikzpicture}
\caption{Models satisfying \ref{H1}, \ref{H2} and \ref{H3}.}
\label{fig:Walks-H1H2H3}
\end{figure} 

\paragraph{Method.}
Working in non-convex cones, here in particular in the three-quarter plane, raises convergence problems. Indeed, it is not difficult to write a functional equation from the properties of harmonicity \ref{P1} and vanishing conditions on the boundary \ref{P2}. However, even if this functional equation \eqref{eq:functional_equation_3/4} seems close to the quarter plane case \eqref{eq:functional_equation_1/4} solved in \cite{Ra-14}, it is in fact fundamentally different: the three quadrants case involves negative powers of $x$ and $y$ making the series not convergent anymore. To remedy this difficulty, we follow the same strategy as in \cite{RaTr-18}: we divide the three quadrants into two symmetric convex cones  and the diagonal (see Figure~\ref{fig:some_sections}) and write a system of two functional equations (one for each cone).  At first sight, this cut increases even more the level of difficulty: we have now two functional equations and more unknowns to deal with, but in the particular case of symmetric harmonic functions and symmetric transition probabilities, the system is composed of twice the same equation and the problem can be seen as a slightly different variation of the quadrant case (see Figure~\ref{fig:change-3/4}). In this symmetric case, we are able to use the tools and methods of \cite{Ra-14}: transform the functional equation into a boundary value problem, solve it and write an explicit expression for the generating function. Finally, let us point out that this method of splitting the domain in two octants can also be applied in the quadrant in the symmetric case, and allow us to give alternative proofs of \cite{Ra-14}.

\paragraph{Main results.}

Our main result is an expression of the diagonal section 
\[
D(x,y)=D_\varphi(xy)=\sum_{i\geq 1}f(i,i)x^{i-1} y^{i-1},
\]
see Lemma~\ref{lem:pol-harmfct}. Let $\widetilde{w}_{\varphi}$ be a conformal mapping (see Section~\ref{sub:ConformalGluingFunction}, $\widetilde{w}_{\varphi}$ depends on the step set and satisfies an algebraic differential equation and can be simply expressed in terms of sine and arcsine functions) and $G_{\varphi}$ an explicit algebraic function defined from $\widetilde{w}_{\varphi}$.  The generating function $D(x,y)$ of discrete harmonic functions \emph{not necessarily positive} (satisfying \ref{P1} and \ref{P2}) can be expressed as 
\begin{equation}
\label{eq:resultD}
D(x,y)=
\displaystyle\frac{P(\widetilde{w}_{\varphi}(xy))}{G_{\varphi}(xy)}, \quad P\in \mathbb{R}[y].
\end{equation}
In particular, taking $P$ of degree 1, we get the \emph{unique positive} discrete harmonic function (satisfying \ref{P1}, \ref{P2} and \ref{P3}), see Theorem~\ref{thm:expression-D-3/4}.

%
%

This expression of $D_{\varphi}(y)$ suffices to give a complete solution to the problem and Theorem \ref{thm:expression-H} gives a formal explicit expression for $H(x,y)$.  This formal series is expressed with the kernel $K(x,y)$, a polynomial in $x$ and $y$, and two conformal mappings  $W_\varphi(x)$ and $\widetilde{W}_\varphi(y)$. Each quantity depends on the transition probabilities of the random walk.

As mentioned earlier, for the simple step set, the positive discrete harmonic function in the half-plane  is given by $f(i,j)=i$ with critical exponent $1$. In the positive quadrant, the positive discrete harmonic function is given by $f(i,j)=ij$ with critical exponent $2$. In this case, the value of the critical exponent in the quadrant (intersection of two half-planes) is twice the value of the critical exponent in the half-plane. Along our study, we find similar relations on harmonic functions in various quadrants (see Lemma~\ref{lem:angle-3/4} and Equation \eqref{eq:angle-1/4}). The growth of harmonic functions can be expressed from the angle $\theta$ defined by \cite[Sec.~1.2]{FaRa-11}
\begin{equation}
\label{eq:theta}
 \theta = \arccos \left( - \frac{\displaystyle \sum_{-1\leq i,j\leq 1} ijp_{i,j}}{\sqrt{\left(\displaystyle\sum_{-1\leq i,j\leq 1}i^{2}p_{i,j}\right)\cdot\left(\displaystyle\sum_{-1\leq i,j\leq 1}j^{2}p_{i,j}\right)}}\right).
\end{equation}
In the three-quarter plane, the critical exponent of the harmonic function is $\pi/(2\pi-\theta)=\pi/(2\theta_{\varphi})$.  After the decomposition of the three quadrants and the changes of variables, the random walk in the three quadrants is equivalent to an inhomogeneous random walk in the half-plane (see Figure~\ref{fig:walks-half-to-quarter-plane}), with associated angle $\theta_{\varphi}=\pi-\theta/2$ in each quadrant (see Figure~\ref{fig:theta-3/4}). We also recover the factor two between harmonic functions in the quadrant and harmonic functions in the half-plane: in the quarter plane, the critical exponent is $\pi/\theta=\pi/(2\theta_{\psi})$. With the same reasoning as the three quadrants case, a random walk associated with angle $\theta$ can be seen as an inhomogeneous random walk in the half-plane, with associated angle $\theta_{\psi}=\theta/2$ in each quadrant (see Figure~\ref{fig:theta-1/4}).

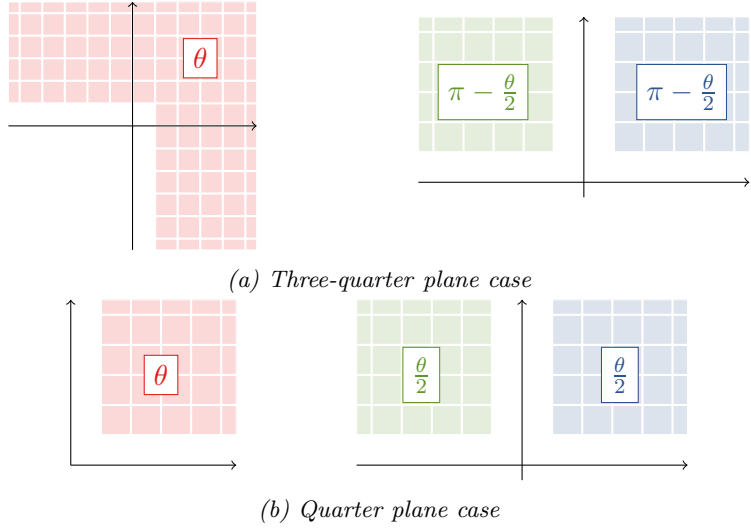
\begin{figure}[t]
\centering
\begin{subfigure}[b]{1\textwidth}
\centering 
\begin{tikzpicture}
\begin{scope}[scale=0.3, yshift=2.5cm]
\draw[white, fill=nred!15] (1,5.5) -- (5.5,5.5) -- (5.5,-5.5) -- (1,-5.5);
\draw[white, fill=nred!15] (1,5.5) -- (-5.5,5.5) -- (-5.5,1) -- (1,1);
\tikzstyle{quadri}=[rectangle,draw,fill=white]
\draw[white, thick] (0,-5.5) grid (5.5,5.5);
\draw[white, thick] (-5.5,0) grid (0,5.5);
\draw[->] (0,-5.5) -- (0,5.5);
\draw[->] (-5.5,0) -- (5.5,0);
\node[nred!90,quadri] at (3,3) {{$\theta$}};
\end{scope}

\begin{scope}[scale=0.4, xshift=15cm]
\tikzstyle{quadri}=[rectangle,draw,fill=white]
\draw[white, fill=dblue!15] (1,1) -- (1,5.5) -- (5.5,5.5) -- (5.5,1);
\draw[white, fill=dgreen!15] (-1,1) -- (-5.5,1) -- (-5.5,5.5) -- (-1,5.5);
\draw[white, thick] (-5.5,0) grid (5.5,5.5);
\draw[->] (0,-0.5) -- (0,5.5);
\draw[->] (-5.5,0) -- (5.5,0);
\node[dblue!90,quadri] at (3.25,3) {{$\pi-\frac{\theta}{2}$}};
\node[dgreen!90, quadri] at (-3.35,3) {{$\pi-\frac{\theta}{2}$}};
\end{scope}
\end{tikzpicture}
\caption{Three-quarter plane case}
\label{fig:theta-3/4}
 \end{subfigure}

\begin{subfigure}[b]{1\textwidth}
\centering
 \begin{tikzpicture}
\begin{scope}[scale=0.4]
\draw[white, fill=nred!15] (1,5.5) -- (5.5,5.5) -- (5.5,1) -- (1,1);
\tikzstyle{quadri}=[rectangle,draw,fill=white]
\draw[white, thick] (0,0) grid (5.5,5.5);
\draw[->] (0,0) -- (0,5.5);
\draw[->] (0,0) -- (5.5,0);
\node[nred!90, quadri] at (3,3) {{$\theta$}};
\end{scope}

\begin{scope}[scale=0.4, xshift=15cm]
\tikzstyle{quadri}=[rectangle,draw,fill=white]
\draw[white, fill=dblue!15] (1,1) -- (1,5.5) -- (5.5,5.5) -- (5.5,1);
\draw[white, fill=dgreen!15] (-1,1) -- (-5.5,1) -- (-5.5,5.5) -- (-1,5.5);
\draw[white, thick] (-5.5,0) grid (5.5,5.5);
\draw[->] (0,-0.5) -- (0,5.5);
\draw[->] (-5.5,0) -- (5.5,0);
\node[dblue!90,quadri] at (3.25,3) {{$\frac{\theta}{2}$}};
\node[dgreen!90, quadri] at (-3.35,3) {{$\frac{\theta}{2}$}};
\end{scope}
\end{tikzpicture}
\caption{Quarter plane case}
\label{fig:theta-1/4}
\end{subfigure}
\caption{A random walk with associated angle $\theta$ is the three quadrants (resp.~the quadrant), can be seen as an inhomogeneous random walk in the half-plane with associated angle $\pi-\frac{\theta}{2}$ (resp.~ $\frac{\theta}{2}$) in each quadrant.}
\end{figure}

\paragraph{Structure of the paper.}
In Section~\ref{sec:Kernel functional equations}, we start with underlining an important tool used in this analytic study of harmonic functions of random walks with small steps. This object, denominated by $K(x,y)$, is called the kernel of the walks and is a polynomial of degree 2 in $x$ and $y$  which encodes the transition probabilities of the walk. Further, we set up various functional equations for generating series of harmonic functions involving $K(x,y)$.

In Section~\ref{sec:Expression for the generating functions}, we express explicitly generating functions of harmonic functions for some models of random walks. To come to these results, we transform a functional equation (built in the previous section) into a boundary value problem in Subsection \ref{subs:BVP}, and the solution process of this problem can be found in Subsection \ref{subs:solution of the boundary value problem}. The resolution of the boundary value problem involves conformal gluing functions, introduced in Subsection \ref{sub:ConformalGluingFunction}. We end this part with the application of the results to the example of the simple random walks (Subsection \ref{subs:example SRW}).

In Appendix~\ref{app:Non symm}, we suppose that we do not have the symmetry of the transition probabilities  of the walks. We are still able to build kernel functional equations, but unfortunately there are too many unknowns and we are not able to solve this difficult problem yet. 

In Appendix~\ref{app:quadrant}, we apply this decomposition of the domain into two convex cones to symmetric harmonic functions of symmetric random walks in the quadrant and find partially the same results as \cite{Ra-14}. We invite the reader to refer to this appendix along the consultation of the main body on the three quadrants.

\paragraph{Acknowledgments.}
I would like to thank Kilian Raschel for his strong support, advice and suggestions. I also would like to thank Samuel Simon for his help with the English. Finally, I thank the anonymous referee for his/her comments and suggestions. 
\\
Most of the work presented in this article was done during my affiliation at Department of Mathematics, Simon Fraser University, Canada \& Institut Denis Poisson, Universit\'e de Tours et Universit\'e d'Orl\'eans, France.

\section{Kernel functional equations}
\label{sec:Kernel functional equations}

\subsection{Kernel of the random walks}

The kernel of the random walks, denoted by $K(x,y)$, is an important object of the study.  It appears in the various functional equations of this section, and we use it later to transform a functional equation into a boundary value problem. The kernel is characterized by the transition probabilities of the associated random walk by
\begin{equation}
\label{eq:Kernel}
	K(x,y)=xy \left[ \sum_{-1\leq i,j \leq 1}p_{i,j}x^{-i} y^{-j} -1 \right].
\end{equation}
The kernel is a polynomial of degree two in $x$ and $y$ (consequence of \ref{H1} and \ref{H3}) and can be written as\footnoteN{We follow the notations of \cite{Ra-14} which are slightly different from the usual ones in \cite{FaIaMa-17}. Indeed, \cite{Ra-14} points out that the kernel in \eqref{eq:Kernel}  is the reciprocal one of the kernel $Q(x,y)$ in \cite{FaIaMa-17}: $K(x,y)=x^{2}y^{2}Q(1/x, 1/y)$.}
\begin{equation}
K(x,y)=\alpha (x)y^{2}+\beta (x)y+\gamma (x) = \widetilde{\alpha} (y) x^{2} + \widetilde{\beta} (y) x+\widetilde{\gamma} (y),
\end{equation}
where
\begin{equation}
\label{coeffnoyau}
\left\{
\begin{array}{l l l}
\alpha (x) = p_{-1,-1}x^{2}+p_{0,-1}x; & \beta (x) = p_{-1,0} x^{2}-x+p_{1,0};& \gamma (x) = p_{0,1}x+p_{1,1};\\
\widetilde{\alpha} (y) =p_{-1,-1}y^{2}+ p_{-1,0}y; & \widetilde{\beta} (y) = p_{0,-1}y^{2}-y+p_{0,1}; & \widetilde{\gamma} (y) = p_{1,0}y+p_{1,1}.
\end{array}
\right.
\end{equation}
We also define the discriminants in the $x$-plane and in the $y$-plane:
\begin{equation}
\label{discriminant}
\widetilde{\delta}(y)=\widetilde{\beta}(y)^2-4\widetilde{\alpha}(y)\widetilde{\gamma}(y), \qquad
\delta (x)=\beta (x)^2-4\alpha (x)\gamma (x).
\end{equation}
The discriminant $\delta (x)$ (resp. $\widetilde{\delta}(y)$) is a polynomial of degree three or four (this is a  consequence of \ref{H2}). Hence there are four branch points $x_1, x_2, x_3, x_4$ (resp. $y_1, y_2, y_3, y_4)$, with $x_4=\infty$ (resp. $y_4=\infty$) when $\delta (x)$ (resp. $\widetilde{\delta}(y)$) is of degree 3. Examples of kernel coefficients are given in Section~\ref{subs:example SRW}.

\begin{lem}[Sec.~2.3 in \cite{FaIaMa-17} and Sec.~2.5 in \cite{Ra-14}] The branch points $x_{i}$ are real and 
\begin{equation}
\left|x_1\right| \leq x_2=1=x_3 \leq \left| x_4 \right| \leq \infty.
\end{equation}
More precisely, we have $x_{1}\in [-1,1)$ and $x_{4}\in (1, \infty) \cup \{\infty\} \cup (-\infty, -1]$. Furthermore, on the real line, $\delta (x)$ is negative if and only if $x\in(x_{1},x_{4})\setminus \{1\}$. Symmetric results hold for the $y_{i}$, roots of $\widetilde{\delta}(y)$.
\end{lem}

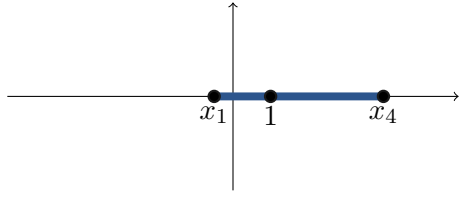
\begin{figure}[t]
\centering
\begin{tikzpicture}
\begin{scope}[scale=0.5]
\draw[->] (-6, 0) -- (6,0);
\draw[->] (0,-2.5) -- (0,2.5); 
\draw[line width=3pt, dblue!90] (-0.5,0) -- (4,0);
\draw[thick, black!90, fill=black] (-0.5,0) circle (0.15 cm);
\node[below] at (-0.5,0) {$x_1$};
\draw[thick, black!90, fill=black] (1,0) circle (0.15 cm);
\node[below] at (1,0) {$1$};
\draw[thick, black!90, fill=black] (4,0) circle (0.15 cm);
\node[below] at (4,0) {$x_4$};
\end{scope}
\end{tikzpicture}
\caption{Cut plane $\mathbb C \setminus \left([x_1,x_4]\right)$}
\label{fig:cut plane}
\end{figure}

Let $Y(x)$  (resp.~$X(y)$) be the multivalued solution to $K(x,Y(x))=0$ (resp.~$K(X(y), y)=0$). These algebraic functions can be written as
\begin{equation}
\label{eq:algebraic curve}
Y(x)=\frac{-\beta(x)\pm \sqrt{\delta(x)}}{2\alpha(x)}
\quad \text{and} \quad
X(y)=\frac{-\widetilde{\beta}(y)\pm \sqrt{\widetilde{\delta}(y)}}{2\widetilde{\alpha}(y)}.
\end{equation}
Over the $x$-plane, the function $Y$ has two branches $Y_{0}$ and $Y_{1}$, both meromorphic on $\mathbb C\setminus [x_{1},x_{4}]$ (see Figure~\ref{fig:cut plane}). We fix the notation by choosing $Y_{0}=Y_{-}$ and $Y_{1}=Y_{+}$. On the whole of $\mathbb C$ (see \cite[Thm~5.3.3]{FaIaMa-17}), 
\begin{equation}
|Y_{0}|\leq |Y_{1}|.
\end{equation}
On the segment $[x_{1},1]$, $Y_{0}(x)$ and $Y_{1}(x)$ are complex conjugate and at the branch points $x_{i}$, we have $Y_{0}(x_{i})=Y_{1}(x_{i})$ (when finite), and this common value is denoted by $Y(x_{i})$.  We introduce the curve $\mathcal L$ by
\begin{equation}
\label{eq:curve_L}
     \mathcal L =Y_0([x_1,1])\cup Y_1([x_1, 1])=\{y\in \mathbb C:
     K(x,y)=0 \text{ and } x\in[x_1,1]\}. 
\end{equation}
We denote by $\mathcal G_{\mathcal L}$ the domain bounded by $\mathcal{L}$ which contains $y_{1}$. Figure~\ref{fig:curves} presents examples of the curve $\mathcal{L}$ for the simple ($p_{1,0}=p_{0,1}=p_{-1,0}=p_{0,-1}=1/4$) and Gouyou-Beauchamps ($p_{1,0}=p_{-1,1}=p_{0,-1}=p_{1,-1}=1/4$) case. 

\begin{lem}[Lem.~6.5.1 in \cite{FaIaMa-17}]
\label{lem:properties_curves}
The curve $\mathcal{L}$ in \eqref{eq:curve_L} is symmetric with respect to the horizontal axis, smooth except at $ Y(1)=1$ where it may have a corner point. At this point, the angle between the curve and the segment $[y_{1},1]$ is given by  by \eqref{eq:theta}.
\end{lem}

\begin{rem}
The quotient $\frac{ \sum_{-1\leq i,j\leq 1} ijp_{i,j}}{\sqrt{\left(\sum_{-1\leq i,j\leq 1}i^{2}p_{i,j}\right)\cdot\left(\sum_{-1\leq i,j\leq 1}j^{2}p_{i,j}\right)}}$ is the coefficient of correlation of the increments of the walk. 
\end{rem}

Similarly, we define the curve $\mathcal M=X_{0}([y_{1},1])\cup X_{1}([y_{1},1])$, and the same symmetric results and notations hold. One has the following automorphism relations (see \cite[Sec.\ 5.3]{FaIaMa-17}):
\begin{equation*}
X_0:\mathcal{G}_{\mathcal L}\setminus [y_1,1]\rightarrow \mathcal{G}_{\mathcal M}\setminus [x_1,1] 
\qquad
\text{and}
\qquad
Y_0:\mathcal{G}_{\mathcal M}\setminus [x_1,1]\rightarrow \mathcal{G}_{\mathcal L}\setminus [y_1,1] 
\end{equation*}
are conformal and inverse of one another. In particular, we have then $X_{0}(Y(x_{1}))=x_{1}$ and $Y_{0}(X(y_{1}))=y_{1}$. 

\begin{figure}[htb]
\centering
\begin{tabular}{cccc}
\begin{tikzpicture}
\begin{scope}[scale=5]
\draw [nred!90, thick] (0.9,0) arc (180:135:0.1);
\draw [dblue!90, thick] plot[smooth] file{X0GB.dat};
\draw [dblue!90, thick] plot[smooth] file{X1GB.dat};
\draw [->] (0,-0.5) -- (0,0.5);
\draw [->] (-0.25,0) -- (1.1,0);
\draw[thick, dblue!90, fill=dblue!90] (1,0) circle (0.005 cm);
\draw [dblue!90] (1,0) node[below]{\scriptsize{$1$}};
\draw[thick, dblue!90, fill=dblue!90] (0.17,0) circle (0.005 cm);
\draw [dblue!90] (0.17,0) node[below]{\scriptsize{$y_1$}};
\draw[thick, dblue!90, fill=dblue!90] (-0.17,0) circle (0.005 cm);
\draw [dblue!90] (-0.25,0) node[below]{\scriptsize{$Y(x_1)$}};
\draw [nred!90] (0.92,0.05) node[left]{\scriptsize{$\theta$}};
\end{scope}
\end{tikzpicture}
\qquad
&
\qquad
\begin{tikzpicture}
\begin{scope}[scale=5, xshift= 5 cm]
\draw [nred!90, thick] (0.25,0) arc (180:90:0.05);
\draw [nred!90] (0.28,0.07) node[left]{\scriptsize{$\theta$}};
\draw[thick, dblue!90] (0,0) circle (0.3 cm);
\draw[thick, dblue!90, fill=dblue!90] (0.3,0) circle (0.005 cm);
\draw [dblue!90] (0.3,0) node[below right]{\scriptsize{$1$}};
\draw[thick, dblue!90, fill=dblue!90] (0.05,0) circle (0.005 cm);
\draw [dblue!90] (0.05,0) node[below]{\scriptsize{$y_1$}};
\draw[thick, dblue!90, fill=dblue!90] (-0.3,0) circle (0.005 cm);
\draw [dblue!90] (-0.38,0) node[below]{\scriptsize{$Y(x_1)$}};
\draw [->] (0,-0.5) -- (0,0.5);
\draw [->] (-0.5,0) -- (0.5,0);
\end{scope}
\end{tikzpicture}
  \end{tabular}
   \caption{The curve $\mathcal{L}$ for Gouyou-Beauchamps model (on the left) and for the simple model (on the right). In the case of Gouyou-Beauchamps model $\theta=\pi/4$, and in the case of the simple model, the curve is simply the unit circle and $\theta=\pi/2$. }
\label{fig:curves}
\end{figure}
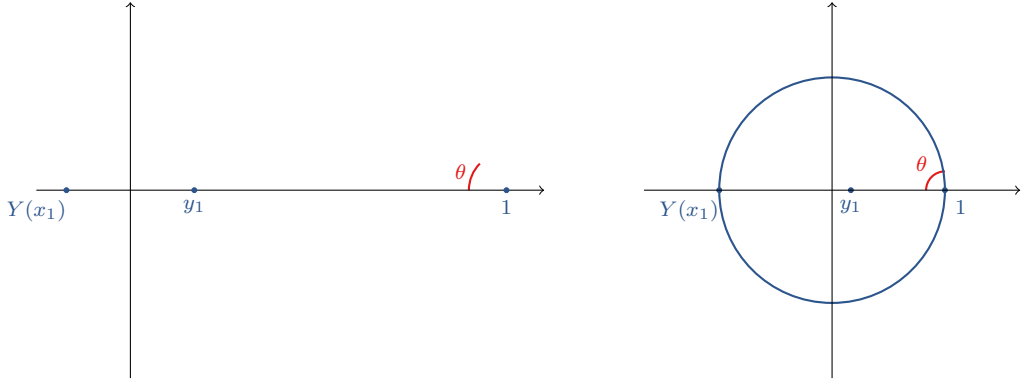

\subsection{A first functional equation}

From the properties \ref{P1} and \ref{P2}, we deduce a functional equation satisfied by the generating function $H(x,y)$ defined in \eqref{eq:generating_function_3/4} for random walks under the hypothesis \ref{H1} 
\begin{equation}
\label{eq:functional_equation_3/4}
K(x,y)H(x,y)=K(x,0)H_{-0}(x^{-1})+K(0,y)H_{0-}(y^{-1})-K(0,0)H_{0,0},
\end{equation}
where 
\begin{multline}
     H_{-0}(x^{-1})= \sum_{i\leq 0} f(i,1)x^{i-1},\quad
     H_{0-}(y^{-1}) = \sum_{j\leq 0} f(1,j)y^{j-1}\quad\text{and}\quad
     H_{0,0}=f(1,1).
\end{multline}

This equation looks similar to the functional equation for discrete harmonic functions of random walks in a quadrant  
\begin{equation}
\label{eq:quarter}
     \mathcal Q=\{(i,j)\in\mathbb Z^2: i\geq1 \text{ and } j\geq1\}
\end{equation}
written in \cite[Sec.\ 2]{Ra-14}. In this case, we assume that random walks satisfy the hypotheses \ref{H1}, \ref{H2}, \ref{H3}, \ref{H4}. We also ask the associated discrete harmonic functions to have the properties \ref{P1q}, \ref{P2q}, \ref{P3q} where 
 
 \begin{enumerate}[label=(\~P\arabic*)]
\item \label{P1q} 
For all $i\geq 1$ and $j \geq 1$, $\widetilde{f}(i,j)=\sum_{-1\leq i_{0},j_{0} \leq 1}p_{i_{0},j_{0}}\widetilde{f}(i+i_{0},j+j_{0})$;
\item \label{P2q} 
If $i\leq 0$ or $j\leq 0$, then $\widetilde{f}(i,j)=0$;
\item \label{P3q} 
If $i>0$ and $j>0$ then $\widetilde{f}(i,j)>0$. 
\end{enumerate} 

The generating function of such harmonic functions $\widetilde{f}$
\begin{equation}
\label{eq:generating_function_1/4}
\widetilde{H}(x,y)=\sum_{i, j\geq 1}\widetilde{f}(i, j)x^{i-1} y^{j-1}
\end{equation}
satisfies the equation
\begin{equation}
\label{eq:functional_equation_1/4}
     K(x,y)\widetilde{H}(x,y)=K(x,0)\widetilde{H}_{-0}(x)+K(0,y)\widetilde{H}_{0-}(y)-K(0,0)\widetilde{H}_{00},
\end{equation}
where
\begin{multline}
\label{eq:sections_1/4}
     \widetilde{H}_{-0}(x)= \sum_{i\geq 1}\widetilde{f}(i,1)x^{i-1},\quad
     \widetilde{H}_{0-}(y)=\sum_{j\geq 1} \widetilde{f}(1,j)y^{j-1}\quad\text{and}\quad
     \widetilde{H}_{0,0}=\widetilde{f}(1,1).
\end{multline}

As noted above, the structure of equations \eqref{eq:functional_equation_3/4} and \eqref{eq:functional_equation_1/4} are the same except that in the case of \eqref{eq:functional_equation_3/4} there are infinitely many terms with positive and negative powers of $x$ and $y$. As noticed in \cite{RaTr-18}, this difference is not anecdotal, the series are not convergent anymore (for instance, it is not possible to evaluate them) and almost all the methodology of \cite{BMMi-10,Ra-12} can no longer be performed.

This difficulty appears as well in the study of counting walks avoiding a quadrant and in order to avoid this situation, in \cite{BM-16} Bousquet-M\'elou views the three-quarter plane as the union of three quadrants and with a combinatorial approach, gets results on the simple and the diagonal walks. In \cite{RaTr-18}, the three quadrants are split in two symmetric convex cones of opening angle $\frac{3\pi}{4}$. In this article, we follow the same strategy as in \cite{RaTr-18}.

\subsection{Functional equations in the  $\frac{3\pi}{4}$-cones}
\label{sub:functional-eq-3pi/4}

The cone $\mathcal C$ is cut into three parts: the lower part $\{i\geq 1, j\leq i-1\}$, the diagonal $\{ i=j\}$ and the upper part $\{j\geq 1, i\geq j-1\}$. Let $L(x,y)$ (resp.\ $D(x,y)$ and $U(x,y)$) be the generating function of harmonic functions evaluated in the lower part (resp.\ diagonal and upper part), see Figure~\ref{fig:some_sections}. By construction, we have
\begin{equation}
\label{eq:equation_cut3parts}
     H(x,y)=L(x,y)+D(x,y)+U(x,y),
\end{equation}
where 
\begin{align*}
L(x,y)=\sum\limits_{\substack{i \geq 1 \\  j\leq i-1}}f(i,j)x^{i-1} y^{j-1}, \quad
D(x,y)=\sum\limits_{\substack{i\geq 1}}f(i,i)x^{i-1} y^{i-1} 
\\
\quad\text{and}\quad
U(x,y)=\sum\limits_{\substack{j \geq 1 \\ i\leq j-1 }}f(i,j)x^{i-1} y^{j-1}.
\end{align*}

\begin{figure}[t]
\centering
\begin{tikzpicture}
\begin{scope}[scale=0.45]
\tikzstyle{quadri}=[rectangle,draw,fill=white]
\draw[white, fill=dblue!15] (1,0) -- (5.5,4.5) -- (5.5,-5.5) -- (1,-5.5);
\draw[white, fill=dgreen!15] (0,1) -- (-5.5,1) -- (-5.5,5.5) -- (4.5,5.5);
\draw[white, thick] (0,-5.5) grid (5.5,5.5);
\draw[white, thick] (-5.5,0) grid (0,5.5);
\draw[Apricot!90, thick] (1,1) -- (5.5,5.5);
\draw[dashed, dgreen!90, thick] (0,1) -- (4.5,5.5);
\draw[dashed, dblue!90, thick] (1,0) -- (5.5,4.5);
\draw[->] (0,-5.5) -- (0,5.5);
\draw[->] (-5.5,0) -- (5.5,0);
\node[dblue!90,quadri] at (3,-3) {$L(x,y)$};
\node[dgreen!90, quadri] at (-3,3) {$U(x,y)$};
\node[Apricot!100, quadri] at (7.4,6.5) {$D(x,y)$};
\node[dgreen!100, quadri, dashed] at (2.8,6.5) {$D^u(x,y)$};
\node[dblue!100, quadri, dashed] at (7.6,4) {$D^\ell(x,y)$};
\end{scope}
\end{tikzpicture}
\caption{Decomposition of the three-quarter plane and associated generating functions}
\label{fig:some_sections}
\end{figure}
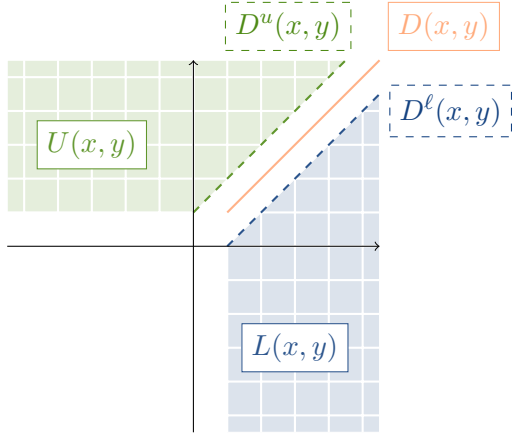

\begin{lem}
\label{lem:functional_equation_sym_3/4}
For any random walks with properties \ref{H1} and \ref{H2}, the generating function $L(x,y)$ satisfies the following functional equation
\begin{multline}
\label{eq:functional_equation_sym}
K(x,y)L(x,y)=-\left(p_{0,1}x+p_{-1,0}x^{2}y+\frac{1}{2}\left(p_{1,1}+p_{-1,-1}x^{2}y^{2}-xy\right)\right)D(x,y)
\\
+\left(p_{1,0}y+p_{1,1}\right)L_{0-}(y^{-1})
+\frac{1}{2}p_{1,1}f(1,1),
\end{multline}
with $L_{0-}(y^{-1})= \sum_{j<0} f(1,j)y^{j-1}$.
\end{lem}

\begin{rem}
\label{rem:symm-cdt}
The symmetry of both the random walks  \ref{H2} and the harmonic function  \ref{P4} are crucial. In fact,  we can directly observe the symmetry of the positive discrete harmonic function in the case of a symmetric random walk. 

Let $S(n)$ be a symmetric random walk in the three quadrants and $\tau_{(i,j)}$ the exit time of the random walk starting at $(i,j)$ from $\mathcal{C}$.  Writing $(p,q)=(r\cos(t), r\sin(t))$, the harmonic function for the Brownian motion in the cone $\{ R\exp(is): 0\leq R < \infty, 0\leq s\leq \frac{3\pi}{2}\}$ of opening angle $\frac{3\pi}{2}$ (see Figure \ref{fig:cone-HF}) is 
\begin{equation*}
u_{\frac{3\pi}{2}}(r,t)=r^{2/3}\sin\left(\frac{2}{3}t\right), \quad t\in \left[ 0, \frac{3\pi}{2}\right]. 
\end{equation*}
After the change of variable $(i,j)\mapsto (-j,i)=\left(r\cos\left(t+\frac{\pi}{2}\right), r\sin\left(t+\frac{\pi}{2}\right)\right)$ (see again Figure \ref{fig:cone-HF}), the harmonic function for the Brownian motion avoiding the negative quadrant is
\begin{equation*}
 u(r,t)=r^{2/3}\sin\left(\frac{2}{3}\left(t+\frac{\pi}{2}\right)\right), \quad t\in\left[-\frac{\pi}{2},\pi\right].
\end{equation*}
 This harmonic function is symmetric in the left cone of Figure~\ref{fig:cone-HF} with respect to the diagonal ($u\left(r,\frac{\pi}{4}-t\right)=u\left(r,\frac{\pi}{4}+t\right)$). Moreover, there is a direct link between the positive discrete harmonic function $V$ and the harmonic function $u$ for the Brownian motion (see~\cite[Lem.~12]{DeWa-15}):
\begin{equation*}
V(i,j)=\lim\limits_{n \rightarrow \infty} \mathbb{E}\left[u((i,j)+S(n)); \tau_{(i,j)}>n\right], \qquad (i,j)\in \mathcal{C}.
\end{equation*} 
The exit time $\tau_{(i,j)}$ and the harmonic function $u$ are  both symmetric in the symmetric cone $\mathcal{C}$. Therefore the positive discrete harmonic function $V$ is symmetric as well. 

The positivity of the harmonic function is crucial to deduce its symmetry from the symmetry of the random walk. Let us give the following counter-example, again in the three quadrants: the \emph{non-symmetric} function $f(i,j)=i^{2}-j^{2}$ is harmonic for the \emph{symmetric} simple walks. Note that this function is \emph{not positive} everywhere in the three quadrants.
\end{rem}

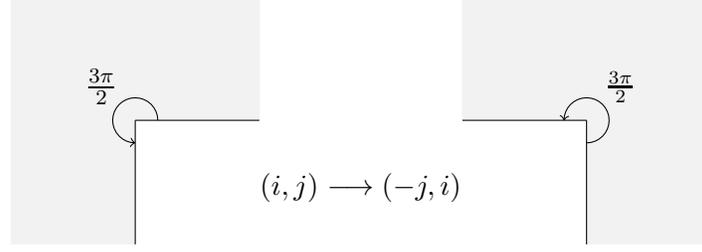
\begin{figure}[t]
\centering
\begin{tikzpicture}
\begin{scope}[scale=0.3]
\draw[gray!10, fill=gray!10] (-5.5,0) -- (-5.5,5.5) -- (5.5,5.5) -- (5.5,0);
\draw[gray!10, fill=gray!10] (-5.5,-5.5) -- (-5.5,0) -- (0,0) -- (0,-5.5);
\draw[] (0,0) -- (0,-5.5);
\draw[] (0,0) -- (5.5,0);
\node[] at (-1.5,1.5) {$\frac{3\pi}{2}$};
\draw [domain=0:270, ->] plot ({cos(\x)}, {sin(\x)});
\end{scope}
\begin{scope}[scale=0.3,xshift=10cm]
\node[] at (0,-3) {$(i,j)\longrightarrow (-j,i)$};
\end{scope}
\begin{scope}[scale=0.3, xshift=20cm]
\draw[gray!10, fill=gray!10] (-5.5,0) -- (-5.5,5.5) -- (5.5,5.5) -- (5.5,0);
\draw[gray!10, fill=gray!10] (5.5,-5.5) -- (5.5,0) -- (0,0) -- (0,-5.5);
\draw [domain=-90:180, ->] plot ({cos(\x)}, {sin(\x)});
\node[] at (1.5,1.5) {\small{$\frac{3\pi}{2}$}};
\draw[] (0,0) -- (0,-5.5);
\draw[] (0,0) -- (-5.5,0);
\end{scope}
\end{tikzpicture}
\caption[Various cones of opening angle $\frac{3\pi}{2}$]{Various cones of opening angle $\frac{3\pi}{2}$, see Remark~\ref{rem:symm-cdt}. On the left the cone $\{ R\exp(is): 0\leq R < \infty, 0\leq s\leq \frac{3\pi}{2}\}$ and the right the cone $\{ R\exp(is): 0\leq R < \infty, -\frac{\pi}{2}\leq s\leq \pi\}$}
\label{fig:cone-HF}
\end{figure}

\begin{proof}[Proof of Lemma~\ref{lem:functional_equation_sym_3/4}]
Thanks to \ref{P1} and \ref{P2}, we can easily write a functional equation for $L(x,y)$ and $D(x,y)$. Let $D^{u}(x,y)$ (resp.\ $D^{\ell}(x,y)$) be the generating function of a harmonic function evaluated in the upper (resp.\ lower) diagonal:
\begin{equation*}
D^{u}(x,y)=\sum_{i\geq 1}f(i-1,i)x^{i-2}y^{i-1}
\qquad \text{and} \qquad
D^{\ell}(x,y)=\sum_{i\geq 1}f(i,i-1)x^{i-1}y^{i-2}.
\end{equation*}
We have
\begin{multline}
\label{eq:fcteq_L}
L(x,y)=\left( \sum_{-1\leq i,j\leq 1}p_{i,j}x^{-i}y^{-j}\right) L(x,y)+ \left(p_{0,1}y^{-1}+p_{-1,0}x\right)D(x,y)-\left(p_{0,-1}y+p_{1,0}x^{-1}\right)D^{\ell}(x,y)
\\
-\left(p_{1,0}x^{-1}+p_{1,1}x^{-1}y^{-1}\right)L_{0-}(y^{-1})+p_{1,0}x^{-1}f(1,0)y^{-1},
\end{multline}
\begin{multline}
\label{eq:fcteq_D-nonsym}
D(x,y)=\left(p_{1,1}x^{-1}y^{-1}+p_{-1,-1}xy\right)D(x,y)-p_{1,1}x^{-1}y^{-1}f(1,1)+\left(p_{1,0}x^{-1}+p_{0,-1}y\right)D^{\ell}(x,y)
\\
-p_{1,0}x^{-1}f(1,0)y^{-1}+\left(p_{0,1}y^{-1}+p_{-1,0}x\right)D^{u}(x,y)-p_{0,1}y^{-1}f(0,1)x^{-1}.
\end{multline}
Due to the symmetry of the cut and the random walks ($p_{i,j}=p_{j,i}$, hypothesis \ref{H2}), we can simplify the last equation and get
\begin{multline}
\label{eq:fcteq_D}
D(x,y)=\left(p_{1,1}x^{-1}y^{-1}+p_{-1,-1}xy\right)D(x,y)-p_{1,1}x^{-1}y^{-1}f(1,1)
\\
+2\left(p_{1,0}x^{-1}+p_{0,-1}y\right)D^{\ell}(x,y)
-2p_{1,0}x^{-1}f(1,0)y^{-1}.
\end{multline}
Plugging \eqref{eq:fcteq_D} into \eqref{eq:fcteq_L}   and multiplying by $xy$, we get \eqref{eq:functional_equation_sym}.
\end{proof}

In Equation \eqref{eq:functional_equation_sym}, the bivariate generating function $L(x,y)$ is related to the bivariate generating function $D(x,y)$ and the univariate generating function $L_{0-}(y^{-1})$. In order to simplify the functional equation \eqref{eq:functional_equation_sym}, we perform the following change of variable 
\begin{equation}
\label{eq:change_var_phi}
     \varphi(x,y)=(xy, x^{-1}).
\end{equation}

Equation \eqref{eq:functional_equation_sym} is transformed into
\begin{align}
\label{eq:functional_equation_octant}
K_{\varphi}(x,y)L_{\varphi}(x,y)=-\left[x\widetilde{\alpha}_{\varphi}(y)+\frac{1}{2}\widetilde{\beta}_{\varphi}(y)\right]D_{\varphi}(y)
+K_{\varphi}(x,0)L_{\varphi}(x,0)+\frac{1}{2}p_{1,1}f(1,1),
\end{align}
with
\begin{equation}
\label{eq:fctgen-changed}
\left\{
\begin{array}{l c l c l}
K\left(\varphi(x,y)\right)&=&\frac{1}{x}K_{\varphi}(x,y),&&\\
L\left(\varphi(x,y)\right)&=&xL_{\varphi}(x,y)&=&x\sum_{i,j\geq 1}f(j, j-i)x^{i-1}y^{j-1},\\
D\left(\varphi(x,y)\right)&=&D_{\varphi}(y)&=&\sum_{i\geq 1}f(i,i)y^{i-1},
\end{array}
\right.
\end{equation}
and 
\begin{multline}
\label{eq:noyauvarphi}
K_{\varphi}(x,y)=\alpha_{\varphi}(x)y^{2}+\beta_{\varphi}(x)y+\gamma_{\varphi}(x)=\widetilde{\alpha}_{\varphi}(y)x^{2}+\widetilde{\beta}_{\varphi}(y)x+\widetilde{\gamma}_{\varphi}(y),\\
\widetilde{\delta}_{\varphi}(y)=\widetilde{\beta}_{\varphi}(y)^2-4\widetilde{\alpha}_{\varphi}(y)\widetilde{\gamma}_{\varphi}(y), \quad
\delta_{\varphi} (x)=\beta_{\varphi} (x)^2-4\alpha_{\varphi} (x)\gamma_{\varphi} (x).
\end{multline}

The functional equation \eqref{eq:functional_equation_octant} now has a closer structure to the functional equation in the quadrant \eqref{eq:functional_equation_1/4}  (see Figure~\ref{fig:change-3/4}), and to solve this problem, we can use some tools stated in \cite{Ra-14}. The significant difference is the mixed factor in $x$ and $y$ in front of $D_{\varphi}(y)$, whereas in the quarter plane case, see Equation \eqref{eq:functional_equation_1/4}, the generating function $\widetilde{H}(x,y)$ is decomposed into a sum of two univariate functions: one section on the $x$-axis, namely $K(x,0)\widetilde{H}(x,0)$, and the other on the $y$-axis, namely $K(0,y)\widetilde{H}(0,y)$. 

After the change of variable $\varphi$, we consider positive harmonic functions associated to random walks in the quadrant (see Figure~\ref{fig:change-3/4}). In \cite[Sec. 2]{Ra-14}, it is proved that for positive discrete harmonic function $\widetilde{f}$ in the quarter plane, the horizontal section $ \widetilde{H}_{-0}(x)=\sum_{i\geq 1} \widetilde{f}(i,1)x^{i-1}$, the vertical section $ \widetilde{H}_{0-}(y)=\sum_{j\geq 1} \widetilde{f}(1,j)y^{j-1}$ and further $K(x,y)\widetilde{H}(x,y)=K(x,y)\sum_{(i, j)\in \mathcal Q} \widetilde{f}(i, j)x^{i-1}y^{j-1}$ are bounded at least on $\{(x,y)\in \mathbb C ^2 : |x|<1, |y|<1\}$. The horizontal section $L_{\varphi}(x,0)$ and the vertical section $D_{\varphi}(y)$ are thus bounded at least in the unit disc. Furthermore, with the functional equation \eqref{eq:functional_equation_octant}, we can deduce that $K_{\varphi}(x,y)L_{\varphi}(x,y)$ is bounded at least for $|x|<1$ and $|y|<1$ as well.

\begin{figure}[t]
\centering
\begin{tikzpicture}
\begin{scope}[scale=0.7]
\tikzstyle{quadri}=[rectangle,draw,fill=white]
\draw[white, fill=dblue!15] (1,1) -- (1,5.5) -- (5.5,5.5) -- (5.5,1);
\draw[white, thick] (0,0) grid (5.5,5.5);
\draw[dashed, dblue!90, thick] (1,1) -- (1,5.5);
\draw[->] (0,0) -- (0,5.5);
\draw[->] (0,0) -- (5.5,0);
\draw[Apricot!90, thick] (0,1) -- (0,5.5);
\node[dblue!90,quadri] at (3,3.5) { $L_{\varphi}(x,y)$};
\node[Apricot!100, quadri] at (0,-1) { $D_{\varphi}(y)$};
\node[dblue!100, quadri, dashed] at (1,6.5) { $D_{\varphi}^\ell(y)$};
\node[dblue!100, quadri] at (7,1) {$L_{\varphi}(x,0)$};
\draw[dblue!100] (1,1) -- (5.5,1);
\end{scope}
\end{tikzpicture}
\caption{Lower convex cone after the change of variable}
\label{fig:change-3/4}
\end{figure}
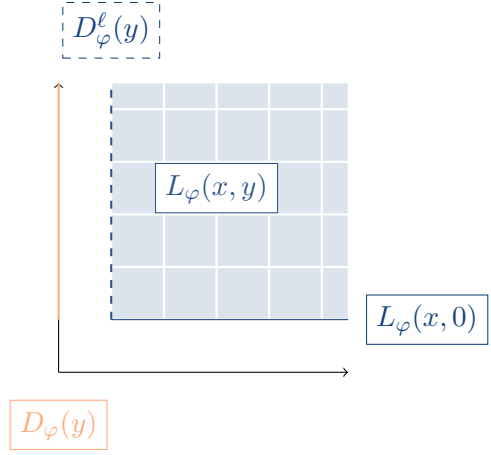

\section{Expression for the generating functions}
\label{sec:Expression for the generating functions}

We start this section with the introduction of an important conformal gluing function $W(x)$ that maps the domain $\mathcal G_{\mathcal M}$ to the complex plane cut along a segment. Then we transform the functional equation \eqref{eq:functional_equation_octant} into a boundary value problem, which is a problem involving both regularity and boundary conditions. Finally we solve this problem and end up with an explicit expression for the generating function $H(x,y)$. The boundary value problem satisfied by the generating function $D_{\varphi}(y)$ and ultimately the generating function $H(x,y)$ will be both expressed in terms of $W$. Further, all quantities with a $\varphi$-tag are objects defined from the changed random walks with kernel $K_{\varphi}(x,y)$, as in Section~\ref{sub:functional-eq-3pi/4}.

\subsection{Conformal gluing function}
\label{sub:ConformalGluingFunction}

\begin{defn}[Conformal gluing functions]
\label{defn:CGF}
A function $W$ is said to be a conformal gluing function for the set $\mathcal{G}_\mathcal M$ if:
\begin{itemize}
	\item $W$ is injective in $\mathcal{G}_\mathcal M$;
     	\item $W$ is analytic on $\mathcal{G}_\mathcal M$;
	\item For all $x$ on $\mathcal M\setminus\{1\}$, $W(x)=W(\overline{x})$.
	\item $W$ is continuous on $\mathcal G_{\mathcal M}\setminus\{1\}$ and $W(1)=\infty$.
\end{itemize}
\end{defn}

\begin{rem}
In fact, to be a conformal gluing function, $W$ only needs to be injective in $\mathcal{G}_\mathcal M$, analytic on $\mathcal{G}_\mathcal M$ and to satisfy $W(x)=W(\overline{x})$ for $x\in \mathcal M$, as soon as $W(x)$ is defined and finite. Moreover, $W$ has an infinite value at least on one point of $\mathcal M$ (or $W$ would be constant, see \cite[Lem.\ 4]{Ra-14}). In this article, we choose $W$ to be infinite at $1$ (if $W_m$ is a conformal gluing function taking the value $\infty$ anywhere else on $\mathcal M$, then we can consider $W=1/(W_m-W_m(1))$).
\end{rem}

The expression of such a function is computed in \cite[Sec.\ 2.2]{FaRa-11}. Let 
\begin{equation}
T(x)=\frac{1}{\sqrt{\frac{1}{3}-\frac{2f(x)}{\delta''(1)}}}, \qquad 
f(x)=
\left\{
\begin{array}{l l}
\frac{\delta''(x_{4})}{6}+\frac{\delta'(x_{4})}{x-x_{4}} & \text{ if } x_{4}\neq \infty, \\
\frac{\delta''(0)}{6}+\frac{\delta'''(0)x}{6} & \text{ if } x_{4}=\infty.
\end{array}
\right.
\end{equation}
Then 
\begin{equation}
\label{eq:CGF}
W(x)=\left[\sin\left(\frac{\pi}{\theta}\left[\arcsin(T(x))-\frac{\pi}{2}\right]\right)\right]^{2},
\end{equation}
with $\theta$ defined in \eqref{eq:theta}. Moreover, there exists $c\neq 0$ such that for $x$ in the neighborhood of 1, 
\begin{equation}
\label{eq:asymt-CGF}
W(x)=\frac{c+o(1)}{(1-x)^{\pi/\theta}}.
\end{equation}

As in \cite[Sec.\ 2.2]{FaRa-11} and \cite[Sec.\ 3.4]{Ra-14}, the expression of $W(x)$  in \eqref{eq:CGF} is valid only for $|T(x)|\leq 1$.  When $|T(x)|\geq 1$, writing $\arcsin(T)=\pi/2\pm i\ln(T+\sqrt{T^{2}-1})$, we can write
\begin{equation}
W(x)=-\frac{1}{4}\left[\left(T(x)+\sqrt{T(x)^2 -1} \right)^{2\pi/\theta}-2+\left( T(x)-\sqrt{T(x)^2 -1}\right)^{2\pi/\theta} \right].
\end{equation}

\begin{lem}
The function $W$ defined in \eqref{eq:CGF} has a zero at $x_{1}$ and is equal to $1$ at $X(y_{1})$. Moreover, for $x\in \mathcal{G_{M}}$, it satisfies the following differential equation
\begin{equation}
\label{eq:diffeqCGF}
\delta(x)W'(x)^{2}=\left(\frac{\pi}{\theta}\right)^{2}\left(\frac{-\delta''(1)}{2}\right)W(x)(1-W(x)).
\end{equation}
\end{lem}

\begin{proof}
We start from \eqref{eq:CGF}. After differentiation and squaring, we can write
\begin{equation}
\label{eq:diffeqCGFproof}
W'(x)^{2}=\left(\frac{\pi}{\theta}\right)^{2}\frac{4T'(x)^{2}}{1-T(x)^{2}}W(x)(1-W(x)).
\end{equation}
The strategy now is to write $\delta(x)$ under its factorized form (remember that $\delta$ has four roots $x_{1}$, $1$, $1$, and $x_{4}$ with possibly $x_{4}=\infty$ when $\delta$ is of degree 3). If $x_{4}=\infty$, then $T(x)^{2}=\frac{1-x_{1}}{1-x}$ and if $x_{4}\neq \infty$, then $T(x)^{2}=\frac{(1-x_{1})(x_{4}-x)}{(x_{4}-x_{1})(1-x)}$. In both cases, $\frac{4T'(x)^{2}}{1-T(x)^{2}}=-\frac{\delta''(1)}{2\delta(x)}$ (and $\delta''(1)< 0$).
We also note that $T(x_{1})=1$ which implies that $W(x_{1})=0$ (and because $W$ is injective, it is the only zero). The only zero of $W'$ is at $X(y_{1})$ (see \cite[Sec.\ 5.3]{BeBMRa-17}), and from \eqref{eq:diffeqCGFproof}, we deduce that $W(X(y_{1}))=1$. 
\end{proof}
Similarly, we define $\widetilde{W}$ a conformal gluing for $\mathcal{G}_{\mathcal L}$ in the sense of Definition \ref{defn:CGF}. We show that $W$ and $\widetilde{W}$ are strongly related as 
\begin{equation*}
\widetilde{W}(Y_{0}(x))=-W(x)+1.
\end{equation*}
Indeed, see \cite[Thm~6 and Rem~6]{Ra-12}\footnoteN{In fact, in \cite[Rem.\ 6]{Ra-12}, it is said that if $W$ is a suitable mapping, then any $\frac{aW(x)+b}{cW(x)+d}$ with $a,b,c,d\in\mathbb C$ such that $ad-bc\neq 0$ is also a suitable mapping. In our case, we are interested in mapping with a pole at $1$. As $Y(1)=1$, by letting $x$ goes to $1$, we get $c=0$. Hence we only need to find $a$ and $b$ such that $ \widetilde{W}(Y(x))=aW(x)+b$.},  $\widetilde{W}(Y(x))$ is a conformal gluing for $\mathcal{G}_{\mathcal M}$ but may not be the one defined by \eqref{eq:CGF}.  However, we know for sure that $ \widetilde{W}(Y(x))=aW(x)+b$, with $a,b\in\mathbb C$. 
By plugging in $x=x_{1}$ and $x=X(y_{1})$, we deduce $a=-1$ and $b=1$. 

Thereafter, we are interested in a conformal gluing function which vanishes at $0$. We consider $w$ defined by 
\begin{equation}
\label{eq:w}
w(x)=W(x)-W(0).
\end{equation}
The function $w$ is a conformal gluing function in the sense of Definition \ref{defn:CGF}, vanishes at $0$, and satisfies the differential equation (which can be easily derived from \eqref{eq:diffeqCGF})
\begin{equation}
\label{eq:diffeqw}
\delta(x)w'(x)^{2}=\left(\frac{\pi}{\theta}\right)^{2}\left(-\frac{\delta''(1)}{2}\right)\left(w(x)+W(0)\right)\left(w(X(y_{1}))-w(x)\right).
\end{equation}

\begin{rem} 
\label{rem:CGF}
Let $\widetilde{w}$ (resp.\ $w_{\varphi}$ and  $\widetilde{w}_{\varphi}$) be a conformal gluing function for $\mathcal G_{\mathcal L}$ (resp.\ $\mathcal G_{\mathcal M_{\varphi}}$ and $\mathcal G_{\mathcal L_{\varphi}}$) which vanishes at $0$. Similar results and properties hold for those functions. 
\end{rem}

The conformal gluing functions $w_{\varphi}$ and $\widetilde{w}_{\varphi}$ are defined from $\theta_{\varphi}$ which can be computed with \eqref{eq:theta} and transition probabilities $\varphi(p_{i,j})=p_{i-j,i}$. The angles $\theta_{\varphi}$ and $\theta$ are simply related (see Figure~\ref{fig:theta-3/4}), as stated in the following lemma. Relations between $\theta$ and $\theta_{\varphi}$ will be ultimately interpreted as relations between the growths of harmonic functions.

\begin{lem}
\label{lem:angle-3/4}
The angle between the curve $\mathcal L_{\varphi}$ and the segment $[y_{\varphi,1},1]$ is given by
\begin{equation}
\label{eq:angle-3/4}
\theta_{\varphi}=\pi-\frac{\theta}{2},
\end{equation}
with $\theta$ defined in \eqref{eq:theta}.
\end{lem}

\begin{proof}
Writing $c=\frac{\sum_{-1\leq i,j \leq 1}ijp_{i,j}}{\sqrt{\left(\sum_{-1\leq i,j\leq 1}i^{2}p_{i,j} \right)\cdot \left( \sum_{-1\leq i,j\leq 1}j^{2}p_{i,j} \right)}}$, we have $\theta= \arccos(-c)$. Similarly, we write $\theta_{\varphi}=\arccos(-c_{\varphi})$. We have
\begin{align*}
c_{\varphi}=\sqrt{\frac{1-c}{2}} 
\Rightarrow
\theta_{\varphi}=\arccos\left(-\sqrt{\frac{1-c}{2}}\right)=\pi-\frac{1}{2}\arccos(-c)=\pi-\frac{\theta}{2}.
\end{align*}
\end{proof}

\subsection{Boundary value problem}
\label{subs:BVP}

\begin{lem}
\label{lem:bvp}
The generating function $D_{\varphi}(y)$ can be analytically continued from the unit disc $\mathcal{D}$ to the domain $\mathcal{G}_{\mathcal L_{\varphi}}\cup\mathcal{D}$ and is continuous on $\overline{\mathcal{G}_{\mathcal L_{\varphi}}}\setminus \{1\}$. Moreover, for all $y\in\mathcal L_{\varphi}\setminus \{1\}$, $D_{\varphi}(y)$ satisfies the following boundary condition
\begin{equation}
\label{eq:bvp_D_3/4}
\frac{g_\varphi(y)}{g_\varphi'(y)}D_{\varphi}(y)-\frac{g_\varphi(\bar{y})}{g_\varphi'(\bar{y})}D_{\varphi}(\bar{y})=0,
\end{equation}
where $g_{\varphi}$ is defined below in \eqref{eq:fct-g}  and  $\widetilde{w}_{\varphi}$ is a conformal gluing function defined in Section~\ref{sub:ConformalGluingFunction}.
\end{lem}

\begin{rem}
The boundary value problem in Lemma~\ref{lem:bvp} is very close to the boundary value problem stated in \cite[Lem.\ 3]{Ra-14}.
\end{rem}

\begin{proof}
We first assume the analyticity and continuity of $D_{\varphi}(y)$ and  we begin to prove the boundary condition \eqref{eq:bvp_D_3/4}. We start to evaluate the functional equation \eqref{eq:functional_equation_octant} at $Y_{0,\varphi}(x)$ for $x$ close to $[x_{\varphi, 1},1)$:
\begin{equation}
\label{eq:fct-3/4-Y0}
0=-\left[x\widetilde{\alpha}_{\varphi}(Y_{\varphi,0}(x))+\frac{1}{2}\widetilde{\beta}_{\varphi}(Y_{\varphi,0}(x))\right]D_{\varphi}(Y_{\varphi,0}(x))
+K_{\varphi}(x,0)L_{\varphi}(x,0)+\frac{1}{2}p_{1,1}f(1,1).
\end{equation}
By letting $x$ go to any point of $[x_{\varphi, 1}, 1]$ with a positive (resp.\ negative) imaginary part, we obtain two new equations. We subtract these two equations and get:
\begin{equation*}
\left[X_{0}(y)\widetilde{\alpha}(y)+\frac{1}{2}\widetilde{\beta}(y)\right]D_{\varphi}(y)-\left[X_{0}(\bar{y})\widetilde{\alpha}(\bar{y})+\frac{1}{2}\widetilde{\beta}(\bar{y})\right]D_{\varphi}(\bar{y}), \quad y\in\mathcal L_{\varphi}\setminus \{1\}.
\end{equation*}
With \eqref{eq:algebraic curve} this last equation can be simplified as
\begin{equation}
\label{eq:bvp_inter}
\sqrt{\widetilde{\delta}_{\varphi}(y)}D_{\varphi}(y)-\sqrt{\widetilde{\delta}_{\varphi}(\bar{y})}D_{\varphi}(\bar{y})=0, \quad y\in\mathcal L_{\varphi}\setminus \{1\}.
\end{equation}
This boundary condition is the homogeneous equation of the boundary condition in \cite[Lem.\ 5]{RaTr-18}, with the same difficulty to deal with a non-meromorphic prefactor in $\mathcal G_{\mathcal L_{\varphi}}$, namely $\sqrt{\widetilde{\delta}_{\varphi}(y)}$ in \eqref{eq:bvp_inter}. 
The idea is to introduce a new function $F_{\varphi}$, analytic in $\mathcal{G}_{\mathcal{L}_\varphi}$ with finite limits on $\mathcal L_{\varphi}$, which satisfies the decoupling condition 
\begin{equation*}
\frac{\sqrt{\widetilde{\delta}_{\varphi}(y)}}{\sqrt{\widetilde{\delta}_{\varphi}(\bar{y})}}=\frac{F_{\varphi}(y)}{F_{\varphi}(\bar{y})}, \quad \forall y\in\mathcal L_{\varphi}\setminus \{1\},
\end{equation*}
which allow us to write the boundary condition \eqref{eq:bvp_inter} as 
\begin{equation*}
F_{\varphi}(y)D_{\varphi}(y)-F_{\varphi}(\bar{y})D_{\varphi}(\bar{y})=0.
\end{equation*}
Such a function $F_{\varphi}$ is defined in \cite[Sec~3.5]{RaTr-18}. It can be constructed from a conformal mapping $g_{\varphi}$ from $\mathcal G_{\mathcal L_{\varphi}}$ onto the unit disc, with the property that $g_{\varphi}(\bar{y})=\overline{g_{\varphi}(y)}$. We have $F_{\varphi}(y)=\frac{g_{\varphi}(y)}{g_{\varphi}'(y)}$ and one can choose $g_\varphi$ defined in $\mathcal G_{\mathcal L_{\varphi}}$ by 
\begin{equation}
\label{eq:fct-g}
g_\varphi(y)=h\left(\frac{2\widetilde{w}_{\varphi}\left(Y_{\varphi}(x_{\varphi, 1}) \right)}{\widetilde{w}_{\varphi}(y)}-1 \right),
\end{equation}
with $h(y)=-y+\sqrt{y^{2}-1}$ and $\widetilde{w}_{\varphi}$ given in Section~\ref{sub:ConformalGluingFunction}. Finally, the boundary condition \eqref{eq:bvp_inter} can be rewritten as \eqref{eq:bvp_D_3/4}

With the same reasoning as \cite[Sec.\ 3.3]{RaTr-18}, we can prove that $D_\varphi(y)$ is analytic in $\mathcal{G}_{\mathcal{L}_\varphi}$, has finite limits on $\mathcal{L}_\varphi\setminus \{1\}$ and then is continuous on $\overline{\mathcal{G}_{\mathcal L_{\varphi}}}\setminus \{1\}$. We start with proving the analytic continuation of $D_{\varphi}(y)$ from the unit disc $\mathcal{D}$ to $\mathcal{G}_{\mathcal{L}_{\varphi}}$. From the functional equation \eqref{eq:functional_equation_octant}, for ${y\in\{ y\in\mathbf{C}: |X_{\varphi,0}(y)|\leq 1\}\cap\mathcal{D}\}}$,
\begin{equation*}
2K_{\varphi}\left( X_{\varphi,0}(y),0\right) L_{\varphi}\left( X_{\varphi,0},0\right) + \sqrt{\widetilde{\delta}(y)}D_{\varphi}(y)+p_{1,1}f(1,1)=0.
\end{equation*}
The latter equation defined on $\mathcal{G}_{\mathcal{L}_\varphi}\cap\mathcal{D}$ can be continued in $\mathcal{G}_{\mathcal{L}_\varphi}\cup\mathcal{D}$. The generating function $D_\varphi(y)$ is analytic on $\mathcal{D}$, and on $\left( \mathcal{G}_{\mathcal{L}_\varphi}\cup\mathcal{D}\right)\setminus \mathcal{D}$, $D_\varphi(y)$ has the same singularities as $\sqrt{\widetilde{\delta}_\varphi(y)}$ and $X_{\varphi,0}(y)$, namely the branch cut $[y_1,y_4]$. However, this segment does not belong to $\left( \mathcal{G}_{\mathcal{L}_\varphi}\cup\mathcal{D}\backslash \mathcal{D}\right)$, then $D_{\varphi}(y)$ can be analytically continued to $\mathcal{G}_{\mathcal{L}_\varphi}$. Then we prove that $D_{\varphi}(y)$ has finite limits on $\mathcal{L}_\varphi\setminus\{ 1\}$. From equation \eqref{eq:fct-3/4-Y0}, we only need to study the zeros of $x\widetilde{\alpha}_{\varphi}\left(Y_{\varphi,0}(x)\right)+\frac{1}{2}\widetilde{\beta}_{\varphi}\left(Y_{\varphi,0} \right)$ for $x$ in $\left[ x_{\varphi,1},x_{\varphi,2}\right]$. Thanks to the relation $X_{\varphi,0}(Y_{\varphi,0}(x))=x$ valid in $\mathcal G_{\mathcal{M}_{\varphi}}$ (see \cite[Cor.~5.3.5]{FaIaMa-17}), it recurs to study the zeros of $\widetilde{\delta}_{\varphi}$ for $y$ in $\left(\mathcal{G}_{\mathcal{L}_{\varphi}}\cup \mathcal D\right) \setminus \mathcal{D}$. The discriminant $\widetilde{\delta}_{\varphi}$ vanishes at $y_{\varphi,1},1,y_{\varphi,4}$ and only $1$ belongs to the last set. Then $D_{\varphi}(y)$ has finite limits on  $\mathcal{L}_\varphi\setminus \{1\}$.

\end{proof}

\subsection{Solution of the boundary value problem}
\label{subs:solution of the boundary value problem}

In this section we solve the boundary value problem stated in Lemma~\ref{lem:bvp} and obtain in Theorem~\ref{thm:expression-D-3/4} an explicit expression for the generating function $D_{\varphi}(y)=\sum_{i\geq 1}f(i,i)y^{i-1}$. Finally, Theorem~\ref{thm:expression-H} gives a formal expression of the generating function $H(x,y)=\sum_{(i,j)\in \mathcal C}f(i,j)x^{i-1}y^{j-1}$. Along the proof Theorem~\ref{thm:expression-D-3/4}, we also state Lemma~\ref{lem:pol-harmfct} which gives an expression of $D_{\varphi}(y)$ for a family of harmonic function non necessarily positive.

\begin{thm}
\label{thm:expression-D-3/4}
The generating function $D_{\varphi}(y)$ can be written as
\begin{equation}
D_{\varphi}(y)=-\frac{f(1,1)}{\widetilde{w}_{\varphi}'(0)}\frac{\pi}{\theta_{\varphi}}\sqrt{-\frac{\widetilde{\delta}_{\varphi}''(1)}{2\widetilde{\delta}_{\varphi}(y)}}\sqrt{1-\widetilde{W}_{\varphi}(0)}\sqrt{\widetilde{W}_{\varphi}(y)},
\end{equation}
with $\theta_{\varphi}$, $\widetilde{w}_{\varphi}(y)$ and $\widetilde{W}_{\varphi}$ defined in Subsection \ref{sub:ConformalGluingFunction} and $\widetilde{\delta}_{\varphi}$ in \eqref{eq:noyauvarphi}.
\end{thm}

\begin{thm}
\label{thm:expression-H}
Let $f$ be a harmonic function associated to a random walk in the three-quarter plane with hypotheses \ref{H1}, \ref{H2}, \ref{H3} and \ref{H4}. The generating function $H(x,y)$ of $f$ can be formally written as the finite sum of convergent generating functions (see~\eqref{eq:equation_cut3parts}) 
\begin{multline}
\label{eq:expression-H}
H(x,y)=-\frac{f(1,1)}{\widetilde{w}_\varphi'(0)}\frac{\pi}{\theta_\varphi}\sqrt{1-\widetilde{W}_\varphi(0)}\sqrt{-\frac{\widetilde{\delta}_\varphi''(1)}{2}}
\left[ \frac{1}{K(x,y)}\left(\frac{\sqrt{1-W_\varphi(y^{-1})}+\sqrt{1-W_\varphi(x^{-1})}}{2}
\right.\right.
\\
\left.\left.
-\frac{\sqrt{\widetilde{W}_\varphi(xy)}}{\sqrt{\widetilde{\delta}_\varphi(xy)}}\left(\left(x^{-1}+y^{-1}\right)\widetilde{\alpha}_\varphi(xy)+\widetilde{\beta}_\varphi(xy)
\right)\right)
+\frac{\sqrt{\widetilde{W}_\varphi(xy)}}{\sqrt{\widetilde{\delta}(xy)}}  \right].
\end{multline}
\end{thm}
\begin{rem}
In equation~\eqref{eq:expression-H}, up to the multiplicative constant $-\frac{f(1,1)}{\widetilde{w}_\varphi'(0)}\frac{\pi}{\theta_\varphi}\sqrt{1-\widetilde{W}_\varphi(0)}\sqrt{-\frac{\widetilde{\delta}_\varphi''(1)}{2}}$, the terms 
\begin{multline*}
\frac{1}{K(x,y)}\left[ \frac{-\left(y^{-1}\widetilde{\alpha}_\varphi(xy)+\frac{1}{2}\widetilde{\beta}_\varphi(xy)\right)}{\sqrt{\widetilde{\delta}_\varphi(xy)}}\sqrt{\widetilde{W}_\varphi(xy)}+\frac{\sqrt{1-W_\varphi(y^{-1})}}{2} \right],
\quad \frac{\sqrt{\widetilde{W}_{\varphi}(xy)}}{\sqrt{\widetilde{\delta}_{\varphi}(xy)}},
\\ 
\text{and} \quad 
\frac{1}{K(x,y)}\left[ \frac{-\left(x^{-1}\widetilde{\alpha}_\varphi(xy)+\frac{1}{2}\widetilde{\beta}_\varphi(xy)\right)}{\sqrt{\widetilde{\delta}_\varphi(xy)}}\sqrt{\widetilde{W}_\varphi(xy)}+\frac{\sqrt{1-W_\varphi(x^{-1})}}{2} \right],
\end{multline*}
contribute respectively for the generating functions $L(x,y)$ in the lower cone $\{ i\geq 1, j\leq i-1\}$,  the diagonal $D(x,y)$ and the generating function $U(x,y)$ in the upper cone $\{ j\geq 1, i\leq j-1\}$. 
\end{rem}

\begin{lem} 
\label{lem:pol-harmfct}
Let $g_{\varphi}$ and $D_{\varphi}$ be as in Lemma~\ref{lem:bvp}.
There exists a polynomial $P\in\mathbb{R} [y]$ such that  
\begin{equation}
\label{eq:pol-harmfct}
\frac{g_\varphi(y)}{g_\varphi'(y)}D_{\varphi}(y)= P(\widetilde{w}_{\varphi}(y)).
\end{equation}
\end{lem}

\begin{rem}
More generally, for any polynomial $P$ of degree $n$, we get a harmonic functions satisfying \ref{P1} and \ref{P2} (but not necessarily \ref{P3}). Remark \ref{rem:pol+neg} illustrates this statement. 
\end{rem}

\begin{proof}[Proof of Lemma~\ref{lem:pol-harmfct}.]
The conformal gluing function $\widetilde{w}_\varphi(y)$ defined in Subsection \ref{sub:ConformalGluingFunction} transforms the domain $\mathcal{G}_{\mathcal{L}_\varphi}$ into the complex plane cut by the segment $I_\varphi=[\widetilde{w}_\varphi(Y_\varphi(x_{\varphi,1})),1)$. Let $\widetilde{w}_\varphi^{-1}$ be the inverse function of $\widetilde{w}_\varphi$. We denote by $\widetilde{w}_\varphi^{-1}(z)^+$ and $\widetilde{w}_\varphi^{-1}(z)^-$ the left and right limits of $\widetilde{w}_\varphi^{-1}$ on $I_\varphi$. The latter are complex conjugate on $I$. Let
\begin{equation*}
\Delta_\varphi(y)=\frac{1}{\frac{g_\varphi(y)}{g_\varphi'(y)}D_\varphi(y)} 
\qquad\text{and}\qquad 
Z_\varphi=\left\{ z\in I_\varphi : \frac{g_\varphi}{g_\varphi'}D_\varphi\circ \widetilde{w}^{-1}(z)=0 \right\}.
\end{equation*}
The set $Z_\varphi$ is finite, otherwise by the principle of isolated zeros, $\frac{g_\varphi(y)}{g_\varphi'(y)}D_\varphi(y)$ would be equal to zero on the whole of $\overline{\mathcal{G}_{\mathcal{L}_\varphi}}$. Let $N_\varphi$ be the cardinal of $Z_\varphi$. With Lemma~\ref{lem:bvp}, we have
\begin{equation*}
\left\{
\begin{array}{l}
\Delta_\varphi\circ \widetilde{w}^{-1}(z)^+-\Delta_\varphi\circ \widetilde{w}^{-1}(z)^-=0 \quad \forall z\in I_\varphi, \\
\Delta_\varphi\circ \widetilde{w}^{-1} \quad \text{ is analytic on } \mathbb{C}\setminus I_\varphi,
\end{array}
\right.
\end{equation*}
which implies that $\Delta_\varphi\circ \widetilde{w}^{-1}$ is holomorphic in $\mathbb{C}$. Furthermore, as $\Delta_\varphi\circ \widetilde{w}_\varphi^{-1}(z)-\displaystyle\sum_{z_0\in Z_\varphi}\frac{1}{(z-z_0)}$ is bounded in $\mathbb{C}$, then by Liouville's theorem, the latter is constant in the whole complex plane. In addition, there exists a polynomial $P$ of degree $N_\varphi$ such that $\left(\frac{g_\varphi}{g_\varphi'}D_\varphi\right)\circ \widetilde{w}_\varphi^{-1}(z)=P(z)$, hence \eqref{eq:pol-harmfct}.
\end{proof}

\begin{proof}[Proof of Theorem~\ref{thm:expression-D-3/4}.]
In particular take $P$ of degree one in \eqref{eq:pol-harmfct}, the  generating function $D_{\varphi}(y)$ can be written as
\begin{equation}
\label{eq:D+const}
D_{\varphi}(y)=\mu \frac{\widetilde{w}_{\varphi}(y)+\nu}{g_\varphi(y)/g_\varphi'(y)}.
\end{equation}
The expression of $\frac{g_\varphi(y)}{g_\varphi'(y)}$ where the function $g$ is defined in \eqref{eq:fct-g} can be simplified into $\frac{g_\varphi(y)}{g_\varphi'(y)}=\frac{\widetilde{w}_{\varphi}(y)\sqrt{\widetilde{w}_{\varphi}(Y_{\varphi}(x_{\varphi,1}))-\widetilde{w}_{\varphi}(y)}}{\widetilde{w}_{\varphi}'(y)\sqrt{\widetilde{w}_{\varphi}(Y_{\varphi}(x_{\varphi,1}))}}$. Furthermore, as the function $w$ is decreasing on $(X(y_{1}),1)$, taking the square root of \eqref{eq:diffeqw} gives
\begin{equation*}
\sqrt{\delta(x)}w'(x)=-\frac{\pi}{\theta}\sqrt{-\frac{\delta''(1)}{2}}\sqrt{w(Y(x_{1}))-w(y)}\sqrt{w(y)+W(0)}.
\end{equation*}
Then, thanks to equation \eqref{eq:asymt-CGF}, there exists a $c_g \neq 0$ such that for $y$ in the neighborhood of 1,
\begin{equation*}
\frac{g_\varphi(y)}{g_\varphi'(y)}=-\frac{\widetilde{w}_\varphi(y)}{\widetilde{w}_\varphi\left(Y_\varphi(x_{\varphi,1}) \right)}\frac{\theta_\varphi}{\pi}\frac{\sqrt{\widetilde{d}_\varphi(y)}}{\sqrt{-\frac{\widetilde{d}_\varphi''(1)}{2}}\sqrt{\widetilde{w}_\varphi(y)+\widetilde{W}_\varphi(0)}}=\frac{c_g+o(1)}{(1-y)^{\pi/(2\theta_\varphi)-1}}.
\end{equation*}
and finally, there exists a $c'_{D}\neq 0$ such that
\begin{equation}
\label{eq:asymp-Diag}
D_{\varphi}(y)= \frac{c'_{D}+o(1)}{(1-y)^{\pi/\theta_{\varphi}-\pi/(2\theta_{\varphi})+1}}=\frac{c'_{D}+o(1)}{(1-y)^{\pi/(2\theta_{\varphi})+1}}=\frac{c'_{D}+o(1)}{(1-y)^{\pi/(2\pi-\theta)+1}}.
\end{equation}
The asymptotics of $D_{\varphi}(y)$ in the last equation \eqref{eq:asymp-Diag} is exactly the asymptotics found in \cite[Eq.~(1.4)]{Mu-19}, see Remark \ref{rem:art-Mustapha}. By unicity of the positive harmonic function, $P$ is a polynomial of degree one.

Going back to \eqref{eq:D+const} and noticing that $g_\varphi(0)/g_\varphi'(0)=0$, we deduce $\nu=-\widetilde{w}_{\varphi}(0)=0$. Then, computing the first term in the expansion of \eqref{eq:D+const} at $y=0$, we find $\mu=\frac{f(1,1)}{\widetilde{w}_{\varphi}'(0)}$.
\end{proof}

\begin{rem}
\label{rem:art-Mustapha}
In \cite{Mu-19}, Mustapha computes the asymptotic of the number of small steps walks excursions in the three-quarter plane.  In particular, his equation (1.4) shows that the polynomial exponent $\alpha_{\mathcal C}$ in the three quadrants can be easily expressed as a function of the critical exponent $\alpha_{\mathcal Q}$ in the quadrant. From this equation we have for $y$ close to $1$, a constant $c_H\neq 0$ such that $H(x,0)=\frac{c_H+o(1)}{(1-x)^{\pi/(2\pi-\theta)}}$, and we deduce that there exists $c_D\neq 0$ such that $D_\varphi(y)=\frac{c_D+o(1)}{(1-y)^{\pi/(2\pi-\theta)+1}}$.  Let us point out that the positivity of the harmonic function (property \ref{P3}) is crucial in the proof of asymptotic results in \cite{Mu-19}.  
\end{rem}


\begin{proof}[Proof of Theorem~\ref{thm:expression-H}.]
 The expression of $D_\varphi(y)$ in Theorem~\ref{thm:expression-D-3/4} suffices to calculate $H(x,y)$. Indeed, to have an expression of $H(x,y)$, we only need to have one of $D(x,y)$ and $L(x,y)$ (thanks to the symmetry of the problem we have $U(x,y)=L(y,x)$). We begin by plugging in $Y_{\varphi,0}(x)$ to the functional equation \eqref{eq:functional_equation_octant} and get
\begin{align}
\label{eq:expression-Lx-3/4}
\nonumber K_\varphi(x,0)L_\varphi(x,0)&=-\frac{1}{2}\left(\sqrt{\widetilde{\delta}_\varphi(Y_{\varphi,0}(x))}D_\varphi(Y_{\varphi,0}(x))+p_{1,1}f(1,1)\right)\\
&=-\frac{f(1,1)}{2}\left(\frac{1}{\widetilde{w}_{\varphi}'(0)}\frac{\pi}{\theta_{\varphi}}\sqrt{-\frac{\widetilde{\delta}_{\varphi}''(1)}{2}}\sqrt{1-\widetilde{W}_\varphi(0)}\sqrt{1-W_\varphi(x)}+p_{1,1}\right).
\end{align}
Together with \eqref{eq:functional_equation_octant}, we can write an expression of $L_{\varphi}(x,y)$ and thanks to the symmetry, we get for free an expression of $U_\varphi(x,y)=L_\varphi(x^{-1},y)$, and with a series expansion of $L_\varphi(x,y)$, $U_{\varphi}(x,y)$ and $D_\varphi(y)$, the value of $f(i,j)$ for all $(i,j)\in\mathcal{C}$. Nevertheless, if we still want to write a formal expression of the generating function $H(x,y)$ defined in \eqref{eq:generating_function_3/4}, we can use the following relations
\begin{equation*}
K_\varphi\left(y^{-1},0\right)L_\varphi(y^{-1},0)=K(0,y)L_{0,-}(y^{-1})
\qquad \text{and} \qquad
D_{\varphi}(xy)=D(x,y),
\end{equation*}
and from them we can write
\begin{multline}
K(x,y)L(x,y)=-\frac{f(1,1)}{\widetilde{w}_\varphi'(0)}\frac{\pi}{\theta_\varphi}\sqrt{1-\widetilde{W}_\varphi(0)}\sqrt{-\frac{\widetilde{\delta}_\varphi''(1)}{2}}
\\
\left[ \frac{-\left(y^{-1}\widetilde{\alpha}_\varphi(xy)+\frac{1}{2}\widetilde{\beta}_\varphi(xy)\right)}{\sqrt{\widetilde{\delta}_\varphi(xy)}}\sqrt{\widetilde{W}_\varphi(xy)}+\frac{\sqrt{1-W_\varphi(y^{-1})}}{2} \right].
\end{multline}
We have now all the ingredients to write the following theorem which gives an expression of $H(x,y)$.
\end{proof}

%

\subsection{Example of the simple random walk}
\label{subs:example SRW}

In this subsection, we apply the results of Subsection \ref{subs:solution of the boundary value problem} 
to the  simple random walk. The change of variable $\varphi$ defined in \eqref{eq:change_var_phi} transforms the simple random walk into the Gessel random walk (see Figure~\ref{fig:SW&GS}). We first need to compute the angle $\theta_\varphi$, the conformal gluing functions $W_\varphi$, $\widetilde{W}_\varphi$ and $\widetilde{w}_\varphi$. Then we get an expression of $D_\varphi(y)$ and $L_\varphi(x,0)$ and end up with a series expansion of $H(x,y)$.

The simple walk is defined by $p_{1,0}=p_{-1,0}=p_{0,-1}=p_{0,1}=\frac{1}{4}$. We easily have $\theta=\frac{\pi}{2}$ and together with \eqref{eq:angle-3/4}, we deduce $\theta_\varphi=\frac{3\pi}{4}$. From Equations \eqref{eq:Kernel} and \eqref{eq:fctgen-changed}, the kernel after the change of variable is defined by
\begin{equation}
K_\varphi(x,y)=xy\left( \frac{1}{4}x+\frac{1}{4}x^{-1}+\frac{1}{4}xy+\frac{1}{4}x^{-1}y^{-1}-1\right),
\end{equation}
\begin{equation}
\label{coeffnoyau-ex}
\left\{
\begin{array}{l l l}
\alpha_\varphi (x) =\frac{x^2}{4}; & \beta_\varphi (x) = \frac{1}{4}x^2-x+\frac{1}{4};& \gamma_\varphi (x) = \frac{1}{4};\\
\widetilde{\alpha}_\varphi (y) =\frac{1}{4}y(y+1); & \widetilde{\beta}_\varphi (y) = -y; & \widetilde{\gamma}_\varphi (y) = \frac{1}{4}(1+y).
\end{array}
\right.
\end{equation}

\begin{figure}[t]
\centering
\begin{tikzpicture}
\begin{scope}[scale=1, xshift=0 cm, yshift=0 cm]
\draw[white, fill=gray!10] (-2,-2) -- (-2,2) -- (2,2) -- (2,-2);
\draw[white, thick] (-2,-2) grid (2,2);
\draw[dblue!90, thick, ->] (0,0)--(1,0) node[above right]{$1/4$};
\draw[dblue!90, thick, ->] (0,0)--(-1,0) node[above left]{$1/4$};
\draw[dblue!90, thick, ->] (0,0)--(0,-1) node[above right]{$1/4$};
\draw[dblue!90, thick, ->] (0,0)--(0,1) node[above left]{$1/4$};
\end{scope}
\begin{scope}[scale=1, xshift=7 cm, yshift=0 cm]
\draw[white, fill=gray!10] (-2,-2) -- (-2,2) -- (2,2) -- (2,-2);
\draw[white, thick] (-2,-2) grid (2,2);
\draw[dblue!90, thick, ->] (0,0)--(1,0) node[above right]{$1/4$};
\draw[dblue!90, thick, ->] (0,0)--(-1,0) node[above left]{$1/4$};
\draw[dblue!90, thick, ->] (0,0)--(-1,-1) node[below left]{$1/4$};
\draw[dblue!90, thick, ->] (0,0)--(1,1) node[above right]{$1/4$};
\end{scope}
\end{tikzpicture}
\caption{Simple random walk (left) and Gessel random walk (right).}
\label{fig:SW&GS}
\end{figure}
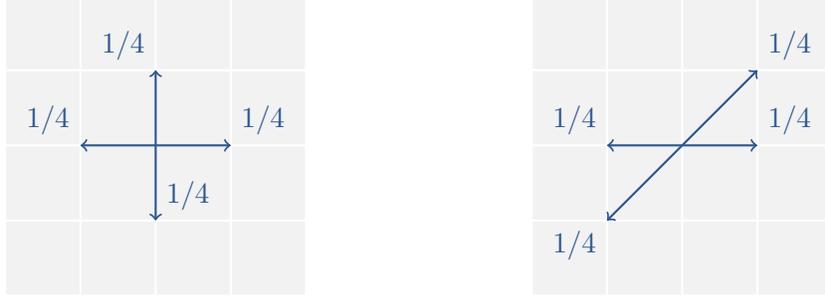 

\paragraph{Conformal gluing functions.}
On one side, we have

\begin{equation*}
\left\{
\begin{array}{l c l}
\delta_\varphi(x)&=&\frac{(x^2-6x+1)(x-1)^2}{16},\\
x_4&=&3+2\sqrt{2},\\
\delta_\varphi'(x_4)&=&3\sqrt{2}+4,\\
\delta_\varphi''(x_4)&=&\frac{11}{2}+3\sqrt{2},\\
\delta_\varphi''(1)&=&-1/2,\\
\end{array}
\right.
\end{equation*}
we have then $T_\varphi(x)=\displaystyle\sqrt{\frac{3+2\sqrt{2}-x}{1-x}}\frac{\sqrt{2}}{2\sqrt{2+\sqrt{2}}}$, and finally 
\begin{multline}
W_\varphi(x)=\frac{1}{4}-\frac{(\sqrt{2}-1)^{4/3}}{16}
\\
\frac{\left(\sqrt{3+2\sqrt{2}-x}+\sqrt{x(3+2\sqrt{2})-1} \right)^{8/3}+\left(\sqrt{3+2\sqrt{2}-x}-\sqrt{x(3+2\sqrt{2})-1} \right)^{8/3}}{\left( 1-x \right) ^{4/3} }.
\end{multline}

On the other side, 
\begin{equation*}
\left\{
\begin{array}{l c l}
\widetilde{\delta}_\varphi(y)&=&-\frac{y(y-1)^2}{4},\\
y_4&=&\infty,\\
\widetilde{\delta}_\varphi''(0)&=&1,\\
\widetilde{\delta}_\varphi'''(0)&=&-3/2,\\
\widetilde{\delta}_\varphi''(1)&=&-1/2,\\
\end{array}
\right.
\end{equation*}
we have then $\widetilde{T}_\varphi(y)=\displaystyle\frac{1}{\sqrt{1-y}}$, and finally
\begin{align}
\widetilde{W}_\varphi(y)&=-\,{\frac { \left( y-2\,\sqrt {y}+1 \right)  \left( 1-\sqrt {y}
 \right) ^{2/3}+ \left( 2\,y-2 \right) \sqrt [3]{1-y}+ \left( \sqrt {
y}+1 \right) ^{2/3} \left( y+2\,\sqrt {y}+1 \right) }{ 4\left( 1-y
 \right) ^{4/3}}}\\
 \widetilde{w}_\varphi(y)&=\widetilde{W}_\varphi(y).
\end{align}

\paragraph{Series expansion of $D_\varphi(y)$ and $L_\varphi(x,0)$.}

Noticing that $\widetilde{w}_\varphi(0)=-16/9$, $\widetilde{W}_\varphi(0)=0$, thanks to Theorem~\ref{thm:expression-D-3/4}, we have 
\begin{align}
D_\varphi(y)&=f(1,1)\frac{3}{8}\frac{\sqrt{\widetilde{W}_\varphi(y)}}{\sqrt{\widetilde{\delta}_\varphi(y)}}
=f(1,1)\left(
1+{\frac{44}{27}}y+{\frac{523}{243}}{y}^{2}+{\frac{17168}{6561}}{y}^{
3}+O
\left( {y}^{4} \right) \right).
\end{align}
Thanks to \eqref{eq:expression-Lx-3/4} with $K_\varphi(x,0)=1/4$, we have 
\begin{align}
L_\varphi(x,0)&=f(1,1)\frac{3}{4}\sqrt{1-W_\varphi(x)}\\\nonumber
&=f(1,1)\left({\frac {3\,\sqrt {3}}{8}}+{\frac{1}{2}}x+ \left( 1-{\frac {\sqrt {3}
}{3}} \right) {x}^{2}+ \left( -{\frac {4\,\sqrt {3}}{3}}+{\frac{145}{
54}} \right) {x}^{3}+O \left( {x}^{4} \right) \right).
\end{align}

\paragraph{Series expansion of $L_\varphi(x,y)$.}

To begin with, we write an expression of $K_\varphi(x,y)L_\varphi(x,y)$.  With \eqref{eq:functional_equation_octant}, we get
\begin{align}
L_\varphi(x,y)&=\frac{1}{K_\varphi(x,y)}\left( -\left[x \frac{1}{4}y(y+1)-\frac{1}{2}y \right]D_\varphi(y)+K_\varphi(x,0)L_\varphi(x,0) \right)\\\nonumber
&=f(1,1)\left[\frac{3\sqrt {3}}{8}+\left(2-\frac{3\sqrt {3}}{8}\right)y+\left(\frac{34}{27}+\frac{3\sqrt {3}}{8}\right)y^2+O \left( {y}^{3} \right) \right.\\\nonumber
&+\left( \frac{1}{2}+\left(-\frac{3}{2}+\frac{3\sqrt {3}}{2}\right)y+\left({\frac{371}{54}}-3\,\sqrt {3} \right) y^2+O \left( {y}^{3} \right)\right)x \\\nonumber
&\left. +\left(1-\frac{\sqrt {3}}{3}+\left( 1-\frac{\sqrt{3}}{24}\right)y+\left({\frac {145\,\sqrt {3}}{24}}-9 \right)y^2 +O \left( {y}^{3} \right) \right)x^2+O \left( {x}^{3} \right)\right].
\end{align}

\begin{rem}
\label{rem:pol+neg}
The harmonic function $f(i,j)=ij$ is a positive harmonic function of the simple walk in the quarter plane. In the three-quarter plane, this function is still harmonic, but does not satisfy \ref{P3}. However, there exists a polynomial $P$ as in \eqref{eq:pol-harmfct} such that $D_\varphi(y)=\sum_{i\geq 1}f(i,i)y^{i-1}=\frac{1+y}{(1-y)^3}$. A quick study of exponent in the same idea of Section~\ref{subs:solution of the boundary value problem} shows that $P$ should be of degree $2$. Then there exists $(a,b,c)\in \mathbb{R}$ such that 
\begin{equation}
\label{eq:exSW-pol}
\frac{g_\varphi(y)}{g_\varphi'(y)}D_\varphi(y)=a\widetilde{w}_\varphi(y)^2+b\widetilde{w}_\varphi(y)+c.
\end{equation}
We have $\frac{g_\varphi(y)}{g_\varphi'(y)}=\frac{\widetilde{w}_\varphi(y)\sqrt{1-\widetilde{w}_\varphi(y)}}{\widetilde{w}_\varphi'(y)}$, and evaluating this last equation with $y=0$ gives $c=0$. Then, dividing by $\widetilde{w}_\varphi(y)$ and letting $y$ go to $0$ gives $b=-\frac{9}{16}$. Finally, examining the first and second term of the expansion of \eqref{eq:exSW-pol} at $y=0$ gives $a=\frac{3}{4}$. Vice versa, letting $P(y)=\frac{3}{4}y^2-\frac{9}{16}y$, when we develop $D_\varphi(y)$ in series, we find
\begin{equation}
D_\varphi(y)=1+4\,y+9\,{y}^{2}+16\,{y}^{3}+25\,{y}^{4}+ O \left( {y}^{5} \right).
\end{equation}

\end{rem}

\footnotesize
\bibliographystyle{plain}
\bibliography{bibl}

\normalsize

\appendix

\section{Non-symmetric case}
\label{app:Non symm}

In this section (except clearly stated) we suppose that the probability transitions of the random walks we consider are not symmetric and a priori neither is the harmonic function. In other words, the random walks satisfy hypotheses \ref{H1}, \ref{H2bis}, \ref{H3} and \ref{H4} with
\begin{enumerate}[label=(\^H\arabic*)]
\setcounter{enumi}{1}
\item \label{H2bis}
We assume that the transition probabilities $p_{0,0}=p_{-1,1}=p_{1,-1}=0$.
\end{enumerate} 
and discrete harmonic functions associated to these random walks satisfy the properties \ref{P1}, \ref{P2}, \ref{P3} but not necessarily \ref{P4}.

Like in Subsection \ref{sub:functional-eq-3pi/4}, we split the three-quadrant into two symmetric convex cones of opening angle $\frac{3\pi}{4}$,  and split the generating function $H(x,y)$ into three generating functions: $L(x,y)$ for harmonic function in the lower part, $D(x,y)$ for the ones on the diagonal and $U(x,y)$ for the ones in the upper part (see \eqref{eq:equation_cut3parts}). Since we do not have the symmetry conditions anymore, instead of one functional equation as in Lemma~\ref{lem:functional_equation_sym_3/4}, we end up with a system of two functional equations.

\begin{lem}
For any random walks with property \ref{H1}, the generating functions $L(x,y)$ and $U(x,y)$ satisfy the following system of functional equations
\begin{equation}
\label{eq:syst_functional_equation}
\left\{
\begin{array}{l c l}
K(x,y)U(x,y)&=&-\left(p_{1,0}y+p_{0,-1}xy^{2}+p_{1,1}+p_{-1,-1}x^{2}y^{2}-xy\right)D(x,y)
\\
&&
+\left(p_{0,1}x+p_{1,1}\right)U_{-0}(x^{-1})-\left(p_{1,0}y+p_{0,-1}xy^{2} \right)D^{\ell}(x,y)
\\
&&
+ p_{1,1}f(1,1)+p_{1,0}f(1,0),
\\
K(x,y)L(x,y)&=&-\left(p_{0,1}x+p_{-1,0}x^{2}y+p_{1,1}+p_{-1,-1}x^{2}y^{2}-xy\right)D(x,y)
\\
&&
+\left(p_{1,0}y+p_{1,1}\right)L_{0-}(y^{-1})-\left(p_{0,1}x+p_{-1,0}x^{2}y\right)D^{u}(x,y)
\\
&&
+ p_{1,1}f(1,1)+p_{0,1}f(0,1).
\end{array}
\right.
\end{equation}

\end{lem}

\begin{rem}
Notice the symmetry between $x$ and $y$ in \eqref{eq:syst_functional_equation}. In the case of a symmetric random walk (\textit{i.e.} if $p_{i,j}=p_{j,i}$) the two functional equations are the same. 
\end{rem}

\begin{proof}
The strategy of the proof is similar to the proof of Lemma~\ref{lem:functional_equation_sym_3/4}. In addition to the functional equations for $L(x,y)$ and $D(x,y)$, see \eqref{eq:fcteq_L} and \eqref{eq:fcteq_D-nonsym}, we can write a functional equation for $U(x,y)$:
\begin{multline}
\label{eq:fcteq_U}
U(x,y)=\left( \sum_{-1\leq i,j\leq 1}p_{i,j}x^{-i}y^{-j}\right) U(x,y)+ \left(p_{1,0}x^{-1}+p_{0,-1}y\right)D(x,y)-\left(p_{-1,0}x+p_{0,1}y^{-1}\right)D^{u}(x,y)
\\
-\left(p_{0,1}y^{-1}+p_{1,1}x^{-1}y^{-1}\right)U_{-0}(x^{-1})+p_{0,1}y^{-1}f(0,1)x^{-1}.
\end{multline}
Mixing equations  \eqref{eq:fcteq_L}, \eqref{eq:fcteq_U} and \eqref{eq:fcteq_D-nonsym} and multiplying by $xy$, we get \eqref{eq:syst_functional_equation}.
\end{proof}

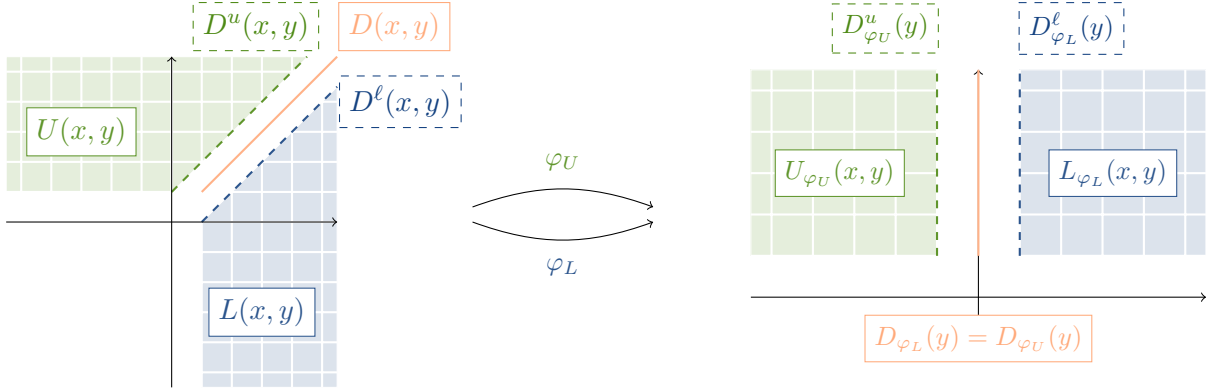
\begin{figure}[t]
\centering
\begin{tikzpicture}
\begin{scope}[scale=0.4, yshift=2.5cm]
\tikzstyle{quadri}=[rectangle,draw,fill=white]
\draw[white, fill=dblue!15] (1,0) -- (5.5,4.5) -- (5.5,-5.5) -- (1,-5.5);
\draw[white, fill=dgreen!15] (0,1) -- (-5.5,1) -- (-5.5,5.5) -- (4.5,5.5);
\draw[white, thick] (0,-5.5) grid (5.5,5.5);
\draw[white, thick] (-5.5,0) grid (0,5.5);
\draw[Apricot!90, thick] (1,1) -- (5.5,5.5);
\draw[dashed, dgreen!90, thick] (0,1) -- (4.5,5.5);
\draw[dashed, dblue!90, thick] (1,0) -- (5.5,4.5);
\draw[->] (0,-5.5) -- (0,5.5);
\draw[->] (-5.5,0) -- (5.5,0);
\node[dblue!90,quadri] at (3,-3) {${L}(x,y)$};
\node[dgreen!90, quadri] at (-3,3) {${U}(x,y)$};
\node[Apricot!100, quadri] at (7.4,6.5) {${D}(x,y)$};
\node[dgreen!100, quadri, dashed] at (2.8,6.5) {${D}^u(x,y)$};
\node[dblue!100, quadri, dashed] at (7.6,4) {${D}^\ell(x,y)$};
\end{scope}

\begin{scope}[scale=0.4, xshift= 13cm]
\draw[->] (-3,3) to[out=20,in=160] (3,3);
\draw [] (0,4.5) node[dgreen!90]{\small {$\varphi_{U}$}};
\draw[->] (-3,2.5) to[out=-20,in=-160] (3,2.5);
\draw [] (0,1) node[dblue!90]{\small {$\varphi_{L}$}};
\end{scope}

\begin{scope}[scale=0.55, xshift=19.5cm]
\tikzstyle{quadri}=[rectangle,draw,fill=white]
\draw[white, fill=dblue!15] (1,1) -- (1,5.5) -- (5.5,5.5) -- (5.5,1);
\draw[white, fill=dgreen!15] (-1,1) -- (-5.5,1) -- (-5.5,5.5) -- (-1,5.5);
\draw[white, thick] (-5.5,0) grid (5.5,5.5);
\draw[dashed, dgreen!90, thick] (-1,1) -- (-1,5.5);
\draw[dashed, dblue!90, thick] (1,1) -- (1,5.5);
\draw[->] (0,-0.5) -- (0,5.5);
\draw[->] (-5.5,0) -- (5.5,0);
\draw[Apricot!90, thick] (0,1) -- (0,5.5);
\node[dblue!90,quadri] at (3.25,3) {\small $L_{\varphi_{L}}(x,y)$};
\node[dgreen!90, quadri] at (-3.35,3) {\small $U_{\varphi_{U}}(x,y)$};
\node[Apricot!100, quadri] at (0,-1) {\small $D_{\varphi_{L}}(y)=D_{\varphi_{U}}(y)$};
\node[dgreen!100, quadri, dashed] at (-2.25,6.5) {\small $D_{\varphi_{U}}^u(y)$};
\node[dblue!100, quadri, dashed] at (2.25,6.5) {\small $D_{\varphi_{L}}^\ell(y)$};
\end{scope}

\end{tikzpicture}
\caption{Decomposition of the three-quarter plane, associated generating functions and changes of variable}
\label{fig:some_sections-1/4}
\end{figure}

In this system of two functional equations, the bivariate generating function $U(x,y)$ (resp.\ $L(x,y)$) is related to the bivariate generating functions $D(x,y)$, $D^{\ell}(x,y)$ (resp.\ $D^{u}(x,y)$) and the univariate generating function $U_{-0}(x^{-1})$ (resp.\ $L_{0-}(y^{-1})$).  In order to simplify this system of functional equations, we perform two changes of variables $\varphi_{L}$ for the lower part and $\varphi_{U}$ for the upper part with
\begin{equation*}
\varphi_{U}(x,y)=(x,x^{-1}y) \qquad \text{and} \qquad \varphi_{L}(x,y)=(xy, x^{-1}).
\end{equation*}
This change of variables transforms the upper part $\{j\geq 1, i\leq j-1\}$ into the left quadrant $\{i\leq -1, j\geq 0\}$ and the lower part $\{i\geq 1, j\leq i-1\}$ into the right quadrant $\{i\geq 1, j\geq 0\}$ (see Figure~\ref{fig:some_sections-1/4}). Note that the diagonal is changed by both $\varphi_{L}$ and $\varphi_{U}$ into the positive $y$-axis (see Figure~\ref{fig:some_sections-1/4}). The step set is changed as well, and walks in the three-quadrant can be seen as inhomogeneous walks in the half-plane with two different step sets in each quadrant and a mixed step set on the positive $y$-axis (see Figure~\ref{fig:walks-quarter-to-half-plane}). Other type of inhomogeneous walks are studied in \cite{BoMuSi-15, BuKa-19}. 

In the symmetric case (symmetry of the random walks and hence the symmetry of the unique positive harmonic function), the study of harmonic functions of  random walks in the three quadrant is then equivalent to the study of harmonic functions of random walks reflected on the $y$-axis and constrained by the $x$-axis  in the positive quadrant (see Figure~\ref{fig:walks-half-to-quarter-plane}). Articles \cite{BeOwRe-18, XuBeOw-19} work on problem in the same vein: their authors study walks in the quadrant with different weights on the boundary.

\begin{figure}[t]
\centering
\begin{tikzpicture}

\begin{scope}[scale=0.45, yshift=4cm]
\draw[white, fill=gray!10] (-5.5,0) -- (-5.5,5.5) -- (5.5,5.5) -- (5.5,0);
\draw[white, fill=gray!10] (0,0) -- (5.5,0) -- (5.5,-5.5) -- (0,-5.5);
\draw[white, thick] (-5.5,0) grid (0,5.5);
\draw[white, thick] (0,-5.5) grid (5.5,5.5);
\draw[->] (0,-5.5) -- (0,5.5);
\draw[->] (-5.5,0) -- (5.5,0);
\draw[nred,thick,->] (3,3) -- (4,3)node[right]{\tiny $p_{1,0}$};
\draw[nred,thick,->] (3,3) -- (4,4)node[above right]{\tiny $p_{1,1}$};
\draw[nred,thick,->] (3,3) -- (3,4)node[above]{\tiny $p_{0,1}$};
\draw[nred,thick,->] (3,3) -- (2,3)node[left]{\tiny $p_{-1,0}$};
\draw[nred,thick,->] (3,3) -- (2,2)node[below]{\tiny $p_{-1,-1}$};
\draw[nred,thick,->] (3,3) -- (3,2)node[below right]{\tiny $p_{0,-1}$};

\draw[nred,thick,->] (0,-3) -- (1,-3)node[right]{\tiny $p_{1,0}$};
\draw[nred,thick,->] (0,-3) -- (1,-2)node[above right]{\tiny $p_{1,2}$};
\draw[nred,thick,->] (0,-3) -- (0,-2)node[above]{\tiny $p_{0,1}$};
\draw[nred,thick,->] (0,-3) -- (0,-4)node[below]{\tiny $p_{0,-1}$};

\draw[nred,thick,->] (-3,0) -- (-4,0);
\draw[nred,thick] (-4,-0.3) node[left]{\tiny $p_{-1,0}$};
\draw[nred,thick,->] (-3,0) -- (-3,1)node[above]{\tiny $p_{0,1}$};
\draw[nred,thick,->] (-3,0) -- (-2,1)node[above right]{\tiny $p_{1,1}$};
\draw[nred,thick,->] (-3,0) -- (-2,0);
\draw[nred,thick] (-2,-0.3) node[right]{\tiny $p_{1,0}$};

\end{scope}

\begin{scope}[scale=0.6, xshift=14cm]
\tikzstyle{quadri}=[rectangle,draw,fill=white]
\draw[white, fill=gray!10] (-7.5,0) -- (7.5,0) -- (7.5,6.5) -- (-7.5,6.5);
\draw[white, thick] (-7.5,0) grid (7.5,6.5);
\draw[->] (0,-0.5) -- (0,6.5);
\draw[->] (-7.5,0) -- (7.5,0);
\draw[dgreen,thick,->] (-5,4) -- (-4,4)node[right]{\tiny $p_{1,0}$};
\draw[dgreen,thick,->] (-5,4) -- (-6,5)node[above left]{\tiny $p_{0,1}$};
\draw[dgreen,thick,->] (-5,4) -- (-5,5)node[above]{\tiny $p_{1,1}$};
\draw[dgreen,thick,->] (-5,4) -- (-6,4)node[left]{\tiny $p_{-1,0}$};
\draw[dgreen,thick,->] (-5,4) -- (-4,3)node[below right]{\tiny $p_{0,-1}$};
\draw[dgreen,thick,->] (-5,4) -- (-5,3)node[below]{\tiny $p_{-1,-1}$};

\draw[dgreen,thick,->] (-5,0) -- (-4,0);
\draw[dgreen,thick] (-4,-0.3) node[right]{\tiny $p_{1,0}$};
\draw[dgreen,thick,->] (-5,0) -- (-6,1)node[above left]{\tiny $p_{0,1}$};
\draw[dgreen,thick,->] (-5,0) -- (-5,1)node[above]{\tiny $p_{1,1}$};
\draw[dgreen,thick,->] (-5,0) -- (-6,0);
\draw[dgreen,thick] (-6,-0.3) node[left]{\tiny $p_{-1,0}$};

\draw[dblue,thick,->] (5,4) -- (6,4)node[right]{\tiny $p_{0,-1}$};
\draw[dblue,thick,->] (5,4) -- (5,5)node[above]{\tiny $p_{1,1}$};
\draw[dblue,thick,->] (5,4) -- (6,5)node[above right]{\tiny $p_{1,0}$};
\draw[dblue,thick,->] (5,4) -- (4,4)node[left]{\tiny $p_{0,1}$};
\draw[dblue,thick,->] (5,4) -- (5,3)node[below]{\tiny $p_{-1,-1}$};
\draw[dblue,thick,->] (5,4) -- (4,3)node[below left]{\tiny $p_{-1,0}$};

\draw[dblue,thick,->] (5,0) -- (6,0);
\draw[dblue,thick] (6,-0.3) node[right]{\tiny $p_{0,-1}$};
\draw[dblue,thick,->] (5,0) -- (5,1)node[above]{\tiny $p_{1,1}$};
\draw[dblue,thick,->] (5,0) -- (6,1)node[above right]{\tiny $p_{1,0}$};
\draw[dblue,thick,->] (5,0) -- (4,0);
\draw[dblue,thick] (4,-0.3) node[left]{\tiny $p_{0,1}$};

\draw[orange!90,thick,->] (0,4) -- (1,4)node[right]{\tiny $p_{0,-1}$};
\draw[orange!90,thick,->] (0,4) -- (0,5)node[above]{\tiny $p_{1,1}$};
\draw[orange!90,thick,->] (0,4) -- (1,5)node[above right]{\tiny $p_{1,0}$};
\draw[orange!90,thick,->] (0,4) -- (-1,5)node[above left]{\tiny $p_{0,1}$};
\draw[orange!90,thick,->] (0,4) -- (0,3)node[below]{\tiny $p_{-1,-1}$};
\draw[orange!90,thick,->] (0,4) -- (-1,4)node[left]{\tiny $p_{-1,0}$};

\draw[orange!90,thick,->] (0,0) -- (1,0);
\draw[orange!90,thick] (1,-0.3) node[right]{\tiny $p_{0,-1}$};
\draw[orange!90,thick,->] (0,0) -- (0,1)node[above]{\tiny $p_{1,1}$};
\draw[orange!90,thick,->] (0,0) -- (1,1)node[above right]{\tiny $p_{1,0}$};
\draw[orange!90,thick,->] (0,0) -- (-1,1)node[above left]{\tiny $p_{0,1}$};
\draw[orange!90,thick,->] (0,0) -- (-1,0);
\draw[orange!90,thick] (-1,-0.3) node[left]{\tiny $p_{-1,0}$};

\end{scope}

\end{tikzpicture}
\caption{Random walks avoiding a quadrant (on the left) can be seen as walks in the half-plane with probability transitions  $\varphi_{U}((p_{i,j})_{-1\leq i,j\leq 1})$ in the left quadrant and $\varphi_{L}((p_{i,j})_{-1\leq i,j\leq 1})$ on the right quadrant (on the right).}
\label{fig:walks-quarter-to-half-plane}
\end{figure}

\begin{figure}[t]
\centering
\begin{tikzpicture}

\begin{scope}[scale=0.6, xshift=0cm]
\tikzstyle{quadri}=[rectangle,draw,fill=white]
\draw[white, fill=gray!10] (-7.5,0) -- (7.5,0) -- (7.5,6.5) -- (-7.5,6.5);
\draw[white, thick] (-7.5,0) grid (7.5,6.5);
\draw[->] (0,-0.5) -- (0,6.5);
\draw[->] (-7.5,0) -- (7.5,0);
\draw[dgreen,thick,->] (-5,4) -- (-4,4)node[right]{\tiny $p_{1,0}$};
\draw[dgreen,thick,->] (-5,4) -- (-6,5)node[above left]{\tiny $p_{1,0}$};
\draw[dgreen,thick,->] (-5,4) -- (-5,5)node[above]{\tiny $p_{1,1}$};
\draw[dgreen,thick,->] (-5,4) -- (-6,4)node[left]{\tiny $p_{-1,0}$};
\draw[dgreen,thick,->] (-5,4) -- (-4,3)node[below right]{\tiny $p_{-1,0}$};
\draw[dgreen,thick,->] (-5,4) -- (-5,3)node[below]{\tiny $p_{-1,-1}$};

\draw[dgreen,thick,->] (-5,0) -- (-4,0);
\draw[dgreen,thick] (-4,-0.3) node[right]{\tiny $p_{1,0}$};
\draw[dgreen,thick,->] (-5,0) -- (-6,1)node[above left]{\tiny $p_{1,0}$};
\draw[dgreen,thick,->] (-5,0) -- (-5,1)node[above]{\tiny $p_{1,1}$};
\draw[dgreen,thick,->] (-5,0) -- (-6,0);
\draw[dgreen,thick] (-6,-0.3) node[left]{\tiny $p_{-1,0}$};

\draw[dblue,thick,->] (5,4) -- (6,4)node[right]{\tiny $p_{-1,0}$};
\draw[dblue,thick,->] (5,4) -- (5,5)node[above]{\tiny $p_{1,1}$};
\draw[dblue,thick,->] (5,4) -- (6,5)node[above right]{\tiny $p_{1,0}$};
\draw[dblue,thick,->] (5,4) -- (4,4)node[left]{\tiny $p_{1,0}$};
\draw[dblue,thick,->] (5,4) -- (5,3)node[below]{\tiny $p_{-1,-1}$};
\draw[dblue,thick,->] (5,4) -- (4,3)node[below left]{\tiny $p_{-1,0}$};

\draw[dblue,thick,->] (5,0) -- (6,0);
\draw[dblue,thick] (6,-0.3) node[right]{\tiny $p_{-1,0}$};
\draw[dblue,thick,->] (5,0) -- (5,1)node[above]{\tiny $p_{1,1}$};
\draw[dblue,thick,->] (5,0) -- (6,1)node[above right]{\tiny $p_{1,0}$};
\draw[dblue,thick,->] (5,0) -- (4,0);
\draw[dblue,thick] (4,-0.3) node[left]{\tiny $p_{1,0}$};

\draw[orange!90,thick,->] (0,4) -- (1,4)node[right]{\tiny $p_{-1,0}$};
\draw[orange!90,thick,->] (0,4) -- (0,5)node[above]{\tiny $p_{1,1}$};
\draw[orange!90,thick,->] (0,4) -- (1,5)node[above right]{\tiny $p_{1,0}$};
\draw[orange!90,thick,->] (0,4) -- (-1,5)node[above left]{\tiny $p_{1,0}$};
\draw[orange!90,thick,->] (0,4) -- (0,3)node[below]{\tiny $p_{-1,-1}$};
\draw[orange!90,thick,->] (0,4) -- (-1,4)node[left]{\tiny $p_{-1,0}$};

\draw[orange!90,thick,->] (0,0) -- (1,0);
\draw[orange!90,thick] (1,-0.3) node[right]{\tiny $p_{-1,0}$};
\draw[orange!90,thick,->] (0,0) -- (0,1)node[above]{\tiny $p_{1,1}$};
\draw[orange!90,thick,->] (0,0) -- (1,1)node[above right]{\tiny $p_{1,0}$};
\draw[orange!90,thick,->] (0,0) -- (-1,1)node[above left]{\tiny $p_{1,0}$};
\draw[orange!90,thick,->] (0,0) -- (-1,0);
\draw[orange!90,thick] (-1,-0.3) node[left]{\tiny $p_{-1,0}$};
\end{scope}

\begin{scope}[scale=0.6, xshift=10cm]
\draw[white, fill=gray!10] (0,0) -- (7.5,0) -- (7.5,6.5) -- (0,6.5);
\draw[white, thick] (0,0) grid (7.5,6.5);
\draw[->] (0,-0.5) -- (0,6.5);
\draw[->] (-0.5,0) -- (7.5,0);

\draw[dpurple,thick,->] (5,4) -- (6,4)node[right]{\tiny $p_{-1,0}$};
\draw[dpurple,thick,->] (5,4) -- (5,5)node[above]{\tiny $p_{1,1}$};
\draw[dpurple,thick,->] (5,4) -- (6,5)node[above right]{\tiny $p_{1,0}$};
\draw[dpurple,thick,->] (5,4) -- (4,4)node[left]{\tiny $p_{1,0}$};
\draw[dpurple,thick,->] (5,4) -- (5,3)node[below]{\tiny $p_{-1,-1}$};
\draw[dpurple,thick,->] (5,4) -- (4,3)node[below left]{\tiny $p_{-1,0}$};

\draw[dpurple,thick,->>] (0,3) -- (1,3)node[right]{\tiny $2p_{-1,0}$};
\draw[dpurple,thick,->>] (0,3) -- (1,4)node[above right]{\tiny $2p_{1,0}$};
\draw[dpurple,thick,->] (0,3) -- (0,2)node[below]{\tiny $p_{-1,-1}$};
\draw[dpurple,thick,->] (0,3) -- (0,4)node[above]{\tiny $p_{1,1}$};

\draw[dpurple,thick,->] (3,0) -- (4,0);
\draw[dpurple,thick] (4,-0.3) node[right]{\tiny ${p}_{-1,0}$};
\draw[dpurple,thick,->] (3,0) -- (2,0);
\draw[dpurple,thick] (2,-0.3) node[left]{\tiny ${p}_{1,0}$};
\draw[dpurple,thick,->] (3,0) -- (4,1)node[above right]{\tiny ${p}_{1,0}$};
\draw[dpurple,thick,->] (3,0) -- (3,1)node[above]{\tiny ${p}_{1,1}$};

\end{scope}

\end{tikzpicture}
\caption{Random walks with symmetric probability transitions avoiding a quadrant can be seen as random walks reflected on the $y$-axis and constrained by the $x$-axis.}
\label{fig:walks-half-to-quarter-plane}
\end{figure}

Unfortunately, due to number of unknown functions, we are not able to solve this system of functional equations yet. Let us end this section by pointing out that the split cone along the diagonal can also be related to the Join-the-Shortest-Queue model (JSQ).  Supposed that there are two lines (see Figure~\ref{fig:JSQ}, left), each of them with a service time exponentially  distributed of rate $r_1$ and $r_2$ and that the customers arrive according to a Poisson process. The clients  choose the shorter queue and if both lines happened to have the same length, the costumers pick one or the other with probability $p_1$ or $p_2$.  A common question in queuing theory is to obtain closed-form expression for the stationary distribution. This JSQ problem can be modeled by random walks in the quarter plane split into two octants, each axis representing the length of each line (see Figure~\ref{fig:JSQ}, right). The symmetric case (when $r_{1}=r_{2}$ and $p_{1}=p_{2}=1/2$) is solvable, see \cite{AdWeZi-91,FoMD-01,KuSu-03} and  \cite[Chap.\ 10]{FaIaMa-17} for reference. On the other side, the non-symmetric case is still an open problem.

\begin{figure}[ht!]
\centering
\begin{tabular}{c c}
\begin{tikzpicture}
\begin{scope}[scale=0.3]
\draw (6.5,1.5) circle (0.3 cm);
\draw[thick, black!90, fill=black!90] (6.5,1.5) circle (0.08 cm);
\draw [black, left] (-7.75,0) node{Arrivals};
\draw [black, right] (9.5,0) node{Departures};
\draw (-5.25,0 ) -- ( -7.25,-2 ) -- (-7.25,2 ) -- cycle;
\shade[left color=dblue!10, right color=dblue!60] (0,2) rectangle (6,5);
\draw (-2,2) rectangle (6,5);
\foreach \x in {-2,-1,...,6} \draw (\x,2) -- (\x,5)  ;
\shade[left color=dgreen!10, right color=dgreen!60] (3,-2) rectangle (6,-5);
\draw (-2,-2) rectangle (6,-5);
\foreach \x in {-2,-1,...,6} \draw (\x,-2) -- (\x,-5)  ;
\draw [->] (-5,0.25) -- (-2.25,3.5);
\draw [->] (-5,-0.25) -- (-2.25,-3.5);
\draw [->] (6.25,3.5) -- (9,0.25);
\draw [->] (6.25,-3.5) -- (9,-0.25);
\end{scope}
\end{tikzpicture}
&
\begin{tikzpicture}
\begin{scope}[scale=0.5]
\draw[white, fill=gray!10] (-1.5,-1.5) -- (-1.5,5.5) -- (5.5,5.5) -- (5.5,-1.5);
\draw[white, thick] (-1.5,-1.5) grid (5.5,5.5);
\draw[] (0,0) -- (5.5,5.5);
\draw[->] (0,-1.5) -- (0,5.5);
\draw[->] (-1.5,0) -- (5.5,0);
\draw[->, nred, semithick] (3,3) -- (4,3);
\draw[->, nred, semithick] (3,3) -- (2,3);
\draw[->, nred, semithick] (3,3) -- (3,4);
\draw[->, nred, semithick] (3,3) -- (3,2);
\node[below right, nred] at (2.8,3.1) {\tiny{$p''_{i,j}$}};
\draw[->,dgreen, semithick] (3,1) -- (3,2);
\draw[->,dgreen, semithick] (3,1) -- (2,1);
\draw[->,dgreen, semithick] (3,1) -- (3,0);
\node[below right, dgreen] at (2.8,0.1) {\tiny{$p'_{i,j}$}};
\draw[->, dblue, semithick] (1,3) -- (0,3);
\draw[->, dblue, semithick] (1,3) -- (1,2);
\draw[->, dblue, semithick] (1,3) -- (2,3);
\node[above left, dblue] at (0.25,2.8) {\tiny{$p_{i,j}$}};
\end{scope}
\end{tikzpicture}
\end{tabular}
\caption{Left: the JSQ model can be represented as a system of two queues, in which the customers choose the shortest one (the green one, on the picture). Right: representation of the JSQ model as an inhomogeneous random walk in the quadrant}
\label{fig:JSQ}
\end{figure}
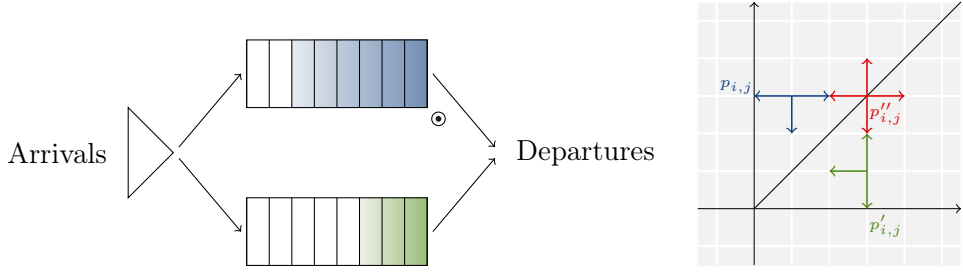

\section{Discrete harmonic functions in the quadrant}
\label{app:quadrant}

Even if discrete harmonic functions of random walks in the quarter plane are already studied in \cite{Ra-14}, we want to point out that the strategy of splitting the domain into three parts (the upper part, the diagonal and the lower part) can also be performed in the quarter plane $\mathcal Q$ defined in \eqref{eq:quarter} and allow us to find explicit expressions of harmonic functions.  We suppose that random walks satisfy the hypotheses \ref{H1}, \ref{H2}, \ref{H3}, \ref{H4}, and associated discrete harmonic functions $\widetilde{f}$ the properties \ref{P1q}, \ref{P2q}, \ref{P3q}. The generating function of such harmonic functions is defined in \eqref{eq:generating_function_1/4}. We split the quadrant $\mathcal Q$ in three parts: the lower part $\{ i\geq 2, 1\leq j\leq i-1\}$, the diagonal and the upper part $\{j\geq 2, 1\leq i\leq j-1\}$. As in \eqref{eq:equation_cut3parts}, by construction we have 
\begin{equation}
\widetilde{H}(x,y)=\widetilde{L}(x,y)+\widetilde{D}(x,y)+\widetilde{U}(x,y),
\end{equation}
where $\widetilde{L}(x,y)=\sum\limits_{\substack{i \geq 2 \\ 1\leq j\leq i-1}}\widetilde{f}(i,j)x^{i-1} y^{j-1}$ denotes the generating function of harmonic functions associated to random walks ending in the lower part, $\widetilde{D}(x,y)=\sum\limits_{\substack{i\geq 1}}\widetilde{f}(i,i)x^{i-1} y^{i-1} $ ending on the diagonal and $\widetilde{U}(x,y)=\sum\limits_{\substack{j \geq 2 \\1\leq i\leq j-1 }}\widetilde{f}(i,j)x^{i-1} y^{j-1}$ in the upper part (see Figure~\ref{fig:some_sections-/4}).

\begin{figure}[t]
\centering
\begin{tikzpicture}
\begin{scope}[scale=0.65]
\tikzstyle{quadri}=[rectangle,draw,fill=white]
\draw[white, fill=dblue!15] (2,1) -- (5.5,4.5) -- (5.5,1);
\draw[white, fill=dgreen!15] (1,2) -- (1,5.5) -- (4.5,5.5);
\draw[white, thick] (0,0) grid (5.5,5.5);
\draw[Apricot!90, thick] (1,1) -- (5.5,5.5);
\draw[dashed, dgreen!90, thick] (1,2) -- (4.5,5.5);
\draw[dashed, dblue!90, thick] (2,1) -- (5.5,4.5);
\draw[->] (0,0) -- (0,5.5);
\draw[->] (0,0) -- (5.5,0);
\node[dblue!90,quadri] at (4,1) {\small $\widetilde{L}(x,y)$};
\node[dgreen!90, quadri] at (1.5,4.5) {\small $\widetilde{U}(x,y)$};
\node[Apricot!100, quadri] at (7.3,6.5) {\small $\widetilde{D}(x,y)$};
\node[dgreen!100, quadri, dashed] at (3.5,6.5) {\small $\widetilde{D}^u(x,y)$};
\node[dblue!100, quadri, dashed] at (7.3,4.5) {\small $\widetilde{D}^\ell(x,y)$};
\end{scope}

\end{tikzpicture}
\caption{Decomposition of the quadrant and associated generating functions}
\label{fig:some_sections-/4}
\end{figure}
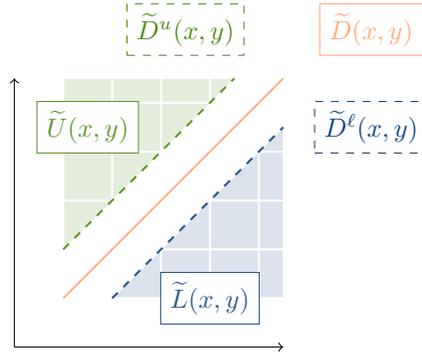

We can write a functional equation for each section and finally get a functional equation in terms of $\widetilde{L}(x,y)$, $\widetilde{L}_{-0}$ and $\widetilde{D}(x,y)$.

\begin{lem} 
\label{lem:functional_equation_sym}
For any random walks with property \ref{H1} and \ref{H2}, the generating function $\widetilde{L}(x,y)$ satisfies the following functional equation
\begin{multline}
\label{eq:functional_equation_sym-/4}
K(x,y)\widetilde{L}(x,y)=-\left(p_{0,1}x+p_{-1,0}x^{2}y+\frac{1}{2}\left(p_{1,1}+p_{-1,-1}x^{2}y^{2}-xy\right)\right)\widetilde{D}(x,y)
\\
+\left(p_{0,1}x+p_{1,1}\right)\widetilde{L}_{-0}(x)
+\left( p_{0,1}x+\frac{1}{2}p_{1,1}\right) \widetilde{f}(1,1),
\end{multline}
with $\widetilde{L}_{-0}(x)= \sum_{i\geq 2} \widetilde{f}(i,1)x^{i-1}$.
\end{lem}

In order to simplify the last functional equation \eqref{eq:functional_equation_sym-/4}, we apply the following change of variables
\begin{equation}
\label{eq:change-var-/4}
\psi(x,y)=(x,x^{-1}y).
\end{equation}
The equation \eqref{eq:functional_equation_sym-/4} is changed into
\begin{equation}
\label{eq:functional_equation_octant/4}
K_{\psi}(x,y)\widetilde{L}_{\psi}(x,y)=-\left[x\widetilde{\alpha}_{\psi}(y)+\frac{1}{2}\widetilde{\beta}_{\psi}(y)\right]\widetilde{D}_{\psi}(y)
+K_{\psi}(x,0)\widetilde{L}_{\psi}(x,0)+\left( p_{0,1}x+\frac{1}{2}p_{1,1}\right) \widetilde{f}(1,1),
\end{equation}
with
\begin{equation}
\left\{
\begin{array}{l c l c l}
K\left(\psi(x,y)\right)&=&\displaystyle\frac{1}{x}K_{\psi}(x,y),&&\\
&&\\
\widetilde{L}\left(\psi(x,y)\right)&=&x\widetilde{L}_{\psi}(x,y)&=&x\displaystyle\sum_{i,j\geq 1}\widetilde{f}(i+j, j)x^{i-1}y^{j-1},\\
&&\\
\widetilde{D}\left(\psi(x,y)\right)&=&\widetilde{D}_{\psi}(y)&=&\displaystyle\sum_{i\geq 1}\widetilde{f}(i,i)y^{i-1},
\end{array}
\right.
\end{equation}
and 
\begin{multline}
\label{eq:noyaupsi}
K_{\psi}(x,y)=\alpha_{\psi}(x)y^{2}+\beta_{\psi}(x)y+\gamma_{\psi}(x)=\widetilde{\alpha}_{\psi}(y)x^{2}+\widetilde{\beta}_{\psi}(y)x+\widetilde{\gamma}_{\psi}(y),\\
\widetilde{\delta}_{\psi}(y)=\widetilde{\beta}_{\psi}(y)^2-4\widetilde{\alpha}_{\psi}(y)\widetilde{\gamma}_{\psi}(y), \quad
\delta_{\psi} (x)=\beta_{\psi} (x)^2-4\alpha_{\psi} (x)\gamma_{\psi} (x).
\end{multline}
The angle $\theta_{\psi}$ is also simply related to $\theta$ (see Figure~\ref{fig:theta-1/4}) as
\begin{equation}
\label{eq:angle-1/4}
\theta_{\psi}=\frac{\theta}{2}.
\end{equation}
The scheme of the proof is the same as the proof of Lemma~\ref{lem:angle-3/4}. Writing $ \theta_{\psi}= \arccos(-c_{\psi})$, we have $c_{\psi}=-\sqrt{\frac{1-c}{2}}$, hence \eqref{eq:angle-1/4}.
The functional equation \eqref{eq:functional_equation_octant/4} is very close to the functional equation in the three-quarter plane \eqref{eq:functional_equation_octant}, the only difference is the additional term $p_{0,1}x$. We can then write a similar boundary value problem as Lemma~\ref{lem:bvp}.

\begin{lem}
\label{lem:bvp-/4}
The generating function $\widetilde{D}_{\psi}(y)$ is analytic in $\mathcal{G}_{\mathcal L_{\psi}}$ and continuous on $\overline{\mathcal{G}_{\mathcal L_{\psi}}}\setminus \{1\}$. Moreover, for all $y\in\mathcal L_{\psi}\setminus \{1\}$, $\widetilde{D}_{\psi}(y)$ satisfies the following boundary condition
\begin{equation}
\label{eq:bvp_D_/4}
\frac{g_\psi(y)}{g_\psi'(y)}\widetilde{D}_{\psi}(y)-\frac{g_\psi(\bar{y})}{g_\psi'(\bar{y})}\widetilde{D}_{\psi}(\bar{y})=0,
\end{equation}
where $\frac{g_\psi(y)}{g_\psi'(y)}=\frac{\widetilde{w}_{\psi}(y)\sqrt{\widetilde{w}_{\psi}(Y_{\varphi}(x_{\psi,1}))-\widetilde{w}_{\psi}(y)}}{\widetilde{w}_{\psi}'(y)\sqrt{\widetilde{w}_{\psi}(Y_{\psi}(x_{\psi,1}))}}$ and  $\widetilde{w}_{\psi}$ is a conformal gluing function defined in Subsection~\ref{sub:ConformalGluingFunction}.

\end{lem}
Moreover, the expression of $\widetilde{D}_{\psi}(y)$ is the same as in the three-quarter plane and we have the following theorem, which is a quarter plane equivalent to Theorem~\ref{thm:expression-D-3/4}.

 \begin{thm}
The generating function $D_{\psi}(y)$ can be written as
\begin{equation}
\label{eq:expression-D-/4}
\widetilde{D}_{\psi}(y)=\frac{\widetilde{f}(1,1)}{\widetilde{w}_{\psi}'(0)}\frac{\pi}{\theta_{\psi}}\sqrt{-\frac{\widetilde{\delta}_{\psi}''(1)}{2\widetilde{\delta}_{\psi}(y)}}\sqrt{\widetilde{w}_{\psi}(Y_{\psi}(x_{\psi,1}))\left( \widetilde{w}_{\psi}(y)+\widetilde{W}_{\psi}(0)\right)},
\end{equation}
with $\theta_{\psi}$, $\widetilde{w}_{\psi}(y)$ and $\widetilde{W}_{\psi}$ defined in Subsection \ref{sub:ConformalGluingFunction} and $\widetilde{\delta}_{\psi}$ in \eqref{eq:noyaupsi}.
\end{thm}
From the expression of $\widetilde{D}_{\psi}$ in \eqref{eq:expression-D-/4} and the functional equation \eqref{eq:functional_equation_octant/4}, we get an expression for $\widetilde{L}_{\psi}(x,0)$:
\begin{equation}
\label{eq:L-SW}
K_{\psi}(x,0)\widetilde{L}_{\psi}(x,0)=-\frac{1}{2}\frac{\widetilde{f}(1,1)}{\widetilde{w}_{\psi}'(0)}\frac{\pi}{\theta_{\psi}}\sqrt{-\frac{\widetilde{\delta}_{\psi}''(1)}{2}}\sqrt{1-\widetilde{W}(0)}\sqrt{1-W_{\psi}(x)}-(p_{0,1}x+\frac{1}{2}p_{1,1})\widetilde{f}(1,1),
\end{equation}
and from \eqref{eq:asymt-CGF} there exists $k\neq 0$ such that, for $x$ in the neighborhood of $1$,
\begin{equation}
K_{\psi}(x,0)\widetilde{L}_{\psi}(x,0)=K(x,0)\widetilde{L}(x,0)=\frac{k+o(1)}{(1-x)^{\pi/(2\theta_{\psi})}}=\frac{k+o(1)}{(1-x)^{\pi/\theta}}.
\end{equation}
This asymptotic results matches that in \cite{Ra-14}. We end this section with the example of the simple random walk. The application $\psi$ changes the simple random walk into the Gouyou-Beauchamps random walk (see Figure~\ref{fig:SW&GB}).
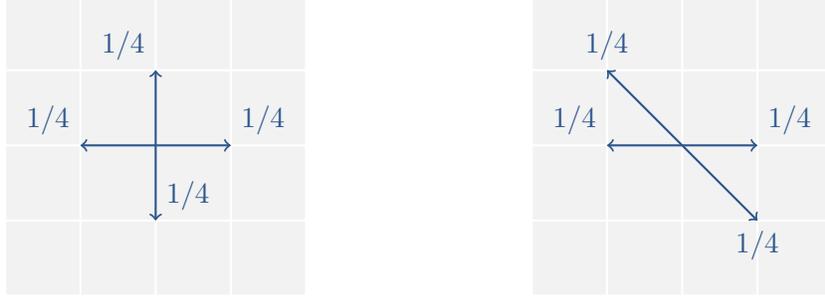
\begin{figure}[t]
\centering
\begin{tikzpicture}
\begin{scope}[scale=1, xshift=0 cm, yshift=0 cm]
\draw[white, fill=gray!10] (-2,-2) -- (-2,2) -- (2,2) -- (2,-2);
\draw[white, thick] (-2,-2) grid (2,2);
\draw[dblue!90, thick, ->] (0,0)--(1,0) node[above right]{$1/4$};
\draw[dblue!90, thick, ->] (0,0)--(-1,0) node[above left]{$1/4$};
\draw[dblue!90, thick, ->] (0,0)--(0,-1) node[above right]{$1/4$};
\draw[dblue!90, thick, ->] (0,0)--(0,1) node[above left]{$1/4$};
\end{scope}
\begin{scope}[scale=1, xshift=7 cm, yshift=0 cm]
\draw[white, fill=gray!10] (-2,-2) -- (-2,2) -- (2,2) -- (2,-2);
\draw[white, thick] (-2,-2) grid (2,2);
\draw[dblue!90, thick, ->] (0,0)--(1,0) node[above right]{$1/4$};
\draw[dblue!90, thick, ->] (0,0)--(-1,0) node[above left]{$1/4$};
\draw[dblue!90, thick, ->] (0,0)--(1,-1) node[below]{$1/4$};
\draw[dblue!90, thick, ->] (0,0)--(-1,1) node[above]{$1/4$};
\end{scope}
\end{tikzpicture}
\caption{Simple random walk (left) and Gouyou-Beauchamps random walk (right).}
\label{fig:SW&GB}
\end{figure} 
In order to compute $\widetilde{L}(x,0)$, we need to calculate $\theta_{\psi}$, $\widetilde{\delta}_{\psi}$, $\widetilde{w}_{\psi}$ and $W_{\psi}$. We easily have $\theta_\psi=\frac{\pi}{4}$ and 
\begin{equation*}
\left\{
\begin{array}{l c l}
\widetilde{\delta}_\psi(y)&=&-\frac{y(y-1)^2}{4},\\
y_4&=&\infty,\\
\widetilde{\delta}_\psi''(0)&=&1,\\
\widetilde{\delta}_\psi'''(0)&=&-3/2,\\
\widetilde{\delta}_\psi''(1)&=&-1/2,\\
\end{array}
\right.
\qquad
\text{we have then }\, \widetilde{W}_\psi(y)=-\frac{16y(y+1)^2}{(1-y)^4}.
\end{equation*}
On the other side, we have
\begin{equation*}
\left\{
\begin{array}{l c l}
\delta_\psi(x)&=&\frac{(x^2-6x+1)(x-1)^2}{16},\\
x_4&=&3+2\sqrt{2},\\
\delta_\psi'(x_4)&=&3\sqrt{2}+4,\\
\delta_\psi''(x_4)&=&\frac{11}{2}+3\sqrt{2},\\
\delta_\psi''(1)&=&-1/2,\\
\end{array}
\right.
\qquad
\text{we have then }\, W_\psi(x)=1-\frac{64x^2}{4(x-1)^4}.
\end{equation*}
Finally, noticing that $K(x,0)=\frac{x}{4}$ and that $\widetilde{H}(x,0)=\widetilde{f}(1,1)+\widetilde{L}(x,0)$, from \eqref{eq:L-SW} we get
\begin{equation}
\widetilde{H}(x,0)=\frac{\widetilde{f}(1,1)}{(1-x)^2}, 
\end{equation}
and this result matches the computation in \cite[Eq.\ 2.6]{Ra-14}.
\end{document}